\documentclass[review]{elsarticle}

\usepackage[abs]{overpic}
\usepackage{graphicx}
\usepackage{subfigure}
\usepackage{stmaryrd}
\usepackage{amsfonts, color}
\usepackage{amsmath}
\usepackage{epsfig}
\usepackage{epstopdf}
\usepackage{psfrag}
\usepackage{ntheorem}

\usepackage[a4paper, total={6in,8in}]{geometry}

\newcommand{\vect}[1]{{\bf #1}}

\newcommand{\bM}{\begin{displaymath}}
\newcommand{\eM}{\end{displaymath}}

\newcommand{\be}{\begin{equation}}
\newcommand{\ee}{\end{equation}}

\newtheorem{lemma}{Lemma}[section]
\newtheorem{theorem}{Theorem}[section]

\newtheorem{remark}{Remark}[section]

\newtheorem*{proof}{Proof}

\journal{Journal of \LaTeX\ Templates}

\graphicspath{{./Second-StokesInterfaxe-figs/}}

\begin{document}

\begin{frontmatter}

\title{Second order convergence of a modified MAC scheme for Stokes interface problems
\tnoteref{mytitlenote}
}
\tnotetext[mytitlenote]{Haixia Dong is partially supported by the National Natural Science Foundation of China (Grant No. 12001193, the Scientific Research Fund of Hunan Provincial Education Department (No.20B376), Changsha Municipal Natural Science Foundation (No. kq2014073). Wenjun Ying is partially supported by the Strategic Priority Research Program of Chinese Academy of Sciences (Grant No. XDA25010405), the National Natural Science Foundation of China (Grant No. DMS-11771290) and the Science Challenge Project of China (Grant No. TZ2016002). }

\author[mymainaddress]{Haixia Dong}
\author[mysecondaryaddress]{Zhongshu Zhao}
\author[mythirdaddress]{Shuwang Li}
\author[myfourthaddress]{Wenjun Ying\corref{mycorrespondingauthor}}
\cortext[mycorrespondingauthor]{Corresponding author}
\ead{wying@sjtu.edu.cn}
\author[myfifthaddress]{Jiwei Zhang}

\address[mymainaddress]{MOE-LCSM, School of Mathematics and Statistics, Hunan Normal University, Changsha, Hunan 410081, P. R. China}
\address[mysecondaryaddress]{School of Mathematical Sciences, and Institute of Natural Sciences, Shanghai Jiao Tong University, Minhang, Shanghai, 200240, P. R. China}
\address[mythirdaddress]{Department of Applied Mathematics, Illinois Institute of Technology, Rettaliata Engineering Center, Room 11B, 10 W. 32nd Street, Chicago, IL60616, USA} 
\address[myfourthaddress]{School of Mathematical Sciences, MOE-LSC, and Institute of Natural Sciences, Shanghai Jiao Tong University, Minhang,
	Shanghai, 200240, P. R. China}	
\address[myfifthaddress]	{School of Mathematics and Statistics, and Hubei Key Laboratory of Computational Science, Wuhan University, Wuhan 430072, China}

\begin{abstract}
Stokes flow equations have been implemented successfully in practice for simulating problems with moving interfaces. Though computational methods produce accurate  solutions and numerical convergence can be demonstrated using a resolution study, the rigorous convergence proofs are usually limited to particular reformulations and boundary conditions. In this paper, a rigorous error analysis of the marker and cell (MAC) scheme for Stokes interface problems with constant viscosity in the framework of the finite difference method is presented. Without reformulating the problem into elliptic PDEs, the main idea is to use a discrete Ladyzenskaja-Babuska-Brezzi (LBB) condition and construct auxiliary functions, which satisfy discretized Stokes equations and possess at least second order accuracy in the neighborhood of the moving interface. In particular, the method, for the first time, enables one to prove second order convergence of the velocity gradient in the discrete $\ell^2$-norm, in addition to the velocity and pressure fields.  Numerical experiments verify the desired properties of the methods and the expected order of accuracy for both two-dimensional and three-dimensional examples.
\end{abstract}


\end{frontmatter}


\section{Introduction}
\setcounter{equation}{0} 
The incompressible Stokes interface problem arises from many
important applications of flows \cite{MR3788553, MR3504551}. For decades,
 numerical methods have been developed for the Stokes interface
problem using grid-based methods (cf. \cite{stenberg1989some,
hansbo2014cut,wang2013hybridizable,mori2008convergence, 
peskin2002immersed, leveque1994immersed, li2006immersed}
and the references therein).  The numerical challenge comes from the low
order of accuracy when computing relevant fields in the neighborhood of the
interface, e.g. first order accuracy in the maximum norm for the immersed
boundary method (IBM). Another numerical issue is the smoothness of the numerical
solution across the interface, either the field function or its gradient.

For past years, the finite difference MAC scheme introduced by Lebedev and Welch \cite{leb1964immersed} has been widely used for solving incompressible Stokes and Navier-Stokes problems \cite{girault1996finite,han1998new,rui2017stability,Li2018NS}. This approach places the the pressure $p$ at the cell center, and the $x-$component velocity $u^{(1)}$ and the $y-$component velocity $u^{(2)}$ at the midpoints of the vertical and horizontal edges of each cell, respectively.   Since Nicolaides and Wu \cite{nicolaides1992analysis,nicolaides1996analysis} first demonstrated the MAC scheme in the form of the finite volume method in 1992, much theoretical analysis has been carried out by interpreting the MAC scheme in different forms, e.g. mixed finite element method \cite{girault1996finite,han1998new}, local discontinuous Galerkin method \cite{kanschat2008divergence}, etc. In most cases, one has only first order accuracy for both velocity and pressure on uniform meshes. On the other hand,  assuming that the pressure has second order accuracy,  Li and Sun \cite{li2015superconvergence} presented stability and second order superconvergence for the MAC scheme of Stokes equations on nonuniform grids.  Later, Rui and Li \cite{rui2017stability} established a discrete LBB condition and gave a rigorous proof of the second order superconvergence for the velocity and pressure fields, some terms of the $H_1$ norm of the velocity on the nonuniform grids. Based on \cite{rui2017stability}, Rui and Li \cite{Li2018dependent,rui2020Darcy, Li2018NS} further extended stability and superconvergence of the MAC scheme for time-dependent Stokes, Stoke-Darcy, and Navier-Stokes problems.  

There exist second order Cartesian grid methods for the Stokes and Navier-Stokes interface problems \cite{leveque1997immersed,lee2003immersed,rutka2008staggered,li2007immersed}. A typical example is the Immersed Interface Methods (IIM), which was proposed to improve the accuracy of IBM. Li  and his collaborators  \cite{MR1860918,MR3623206,Li2015Analysis,MR3788553,Tong2020gradient} have done  a series of works on the proof of convergence for the elliptic and Stokes interface problems in the past decades. For example, Hu and Li \cite{Hu2018IIM} gave rigorous error analysis of the augmented IIM (AIIM) for Stokes interface problems, in which second order accuracy for both velocity and pressure are established under the assumption that an auxiliary, second order accurate, Neumann boundary condition for pressure is provided. 
 Considering an elliptic interface problem, Tong and Wang et al. \cite{Tong2020gradient} proposed a new strategy based on IIM to confirm the second order convergence for 1D problems theoretically and nearly second order convergence for 2D problems except for a factor of $|\log h|$ of the gradient numerically. The main idea of this method is that the gradient at both regular and irregular grid points (also on the interface) is computed using the interpolation from the solution at grid points obtained from IIM.
Specifically, by introducing augmented variables,  Tan et al. \cite{tan2008immersed, tan2009immersed,tan2011implementation} used IIM with the MAC scheme to solve two-phase incompressible Stokes equations, 
which numerically produce second order accuracy for velocity and nearly second-order accuracy for pressure. 
Later, a direct IIM approach based on the MAC scheme \cite{Chen2018A}  was proposed for 2D two-phase Stokes flow, 
 which has also demonstrated its success in capturing non-smooth velocity and pressures. This approach is easy to implement, and is computationally efficient.  
Recently, a sharp capturing method with MAC scheme \cite{wang2022simple} was presented for two-phase incompressible Navier-Stokes equations. This method is of first-order accuracy for velocity and pressure.
 However, relatively less work is done to strictly show the accuracy for both velocity and pressure computed by the MAC scheme.


The MAC scheme has the advantages of simplicity, effectiveness, and ability to use existing fast solvers. But the accuracy for the gradient of the velocity is also needed in many situations, the second order accuracy of the gradient for the MAC scheme is not rigorously proved until now.
The purpose of this paper is to establish and  analyze a second-order finite difference MAC scheme for the Stokes interface problem.  The main contributions include
\begin{itemize}
\item[1)] A modified finite difference MAC scheme is constructed. To resolve the jumps in the solution and its derivatives sharply,  Mayo's technique \cite{MR736332, MR1145178} is used to incorporate the jumps into the MAC scheme near the interface. It is noteworthy that the technique to compute the jump conditions and calculate the correction terms is essentially different from that in \cite{tan2009immersed}. The computation is accomplished along the direction of the Cartesian grid line. 
\item[2)] By establishing discrete auxiliary functions, which depend on the exact velocity or pressure and discretizing parameters $h$, second order accuracy between these functions and the approximate numerical solutions (velocity, pressure and the gradient of velocity) of the modified MAC scheme is achieved. The auxiliary functions satisfy the discrete equations and cancel lower order truncation errors near the interface and boundaries. As a result, the truncation errors are of second order accuracy at all grid points consisting of internal regular points, boundary regular points and irregular points. Though this idea has been used for initial boundary value problems \cite{MR0166942}, such as Navier-Stokes problems \cite{MR1163348, MR1220643}, this is the first time developed for the interface problem.  
\item[3)]  On account of the good approximability of auxiliary functions to numerical solutions, second order $\ell^2$-accuracy in the velocity, the pressure as well as the gradient of the velocity of the modified MAC scheme is rigorously proved. Note that the convergence analysis of the gradient is very challenging yet and very limited results are available along this line in addition to some results for elliptic interface problems \cite{beale2007accuracy,dong2016unfitted,Tong2020gradient}. To the best knowledge of us,  this is the first work to analyze second order convergence for the modified MAC scheme.  
\end{itemize}

Recall the major challenge comes from the fact that the truncation errors on the boundaries are the order of $\mathcal{O}(1)$ and only the first order near the interface. Unlike the three-Poisson-equation decomposition approach \cite{Hu2018IIM} that second order convergence of pressure and velocity has been shown under some assumptions for auxiliary Neumann boundary condition, second  order accuracy by means of a discrete LBB condition and using the above established auxiliary functions is achieved.  It is worth mentioning that the scheme and analysis are only given for two dimensional problems, but they can be extended to three dimensional problems. In fact, the numerical accuracy is verified using 3D examples.

The remainder of this paper is organized as follows. Section 2 introduces
the model problem and its variational formulation. Section 3 describes
the modified MAC scheme for the Stokes interface problem. Section 4 presents
error analysis for the numerical solutions. Section 5 shows numerical examples 
to validate the theoretical results. Section 6 gives some concluding remarks.

\section{The Model Problem}
\label{sec;notation}
Let $\Omega$ be a two dimensional rectangular domain, and $\Omega^+\subset\subset\Omega$ be
a simply connected domain with smooth boundary $\Gamma$. Set $\Omega^-=\Omega\backslash\bar{\Omega}^+$ and consider the following Stokes interface problem 
\begin{equation}
\label{interfaceP}
\begin{split}
-\mu \Delta\vect   u+\nabla p&={\vect   f},  \,\quad\;\;\hbox{in}\; \Omega^+\cup\Omega^-,\\
\nabla\cdot \vect   u&=0,\quad\;\;\, \hbox{in}\; \Omega^+\cup\Omega^-,\\
[\![ \vect   u ]\!]&=\vect 0, \quad\;\;\,\hbox{on}\; \Gamma,\\
[\![ \pmb \sigma(\vect   u, p)\vect   n]\!]&=\pmb \psi, \;\quad\;\hbox{on}\; \Gamma,\\
\vect   u&=\vect   u_b, \quad\; \hbox{on}\; \partial \Omega,
\end{split}
\end{equation}
with a constant viscosity $\mu$. Here, $\vect   u=(u^{(1)}, u^{(2)})^T$, $p$ and $\vect   f=(f^{(1)}, f^{(2)})^T$ represent the velocity, pressure and external force, respectively. The stress tensor is defined by 
$$\pmb\sigma(\vect   u, p)
=-p\vect   I+\mu(\nabla \vect   u+(\nabla \vect   u)^T),$$ and $\vect   n$ represents the unit normal vector on $\Gamma$ pointing from $\Omega^+$ to $\Omega^-$. One can refer to Fig. \ref{area} for illustration.
The jump notation across the interface $\Gamma$ is denoted by $[\![ \vect   v]\!]=\vect   v^+-\vect   v^-$ with $v^+$ and $v^-$ be respectively the limit values of $v$ on two sides of the interface. 
It is known that due to the incompressibility constraint,  the boundary data $\vect   u_b$ should satisfy the following compatibility condition
\begin{equation*}
\int_{\partial \Omega} \vect   u_b\cdot \vect   n_b\,ds =0,
\end{equation*}
where $\vect   n_b$ is the outer unit normal on $\partial\Omega$.  In this paper, for simplicity of analysis, assume that $\mu=1$ and $\vect   u_b=\vect 0$, but non-homogeneous boundary conditions will be considered in the numerical examples.
\begin{figure}[ht!]
\centering
\includegraphics[width=0.75\textwidth]{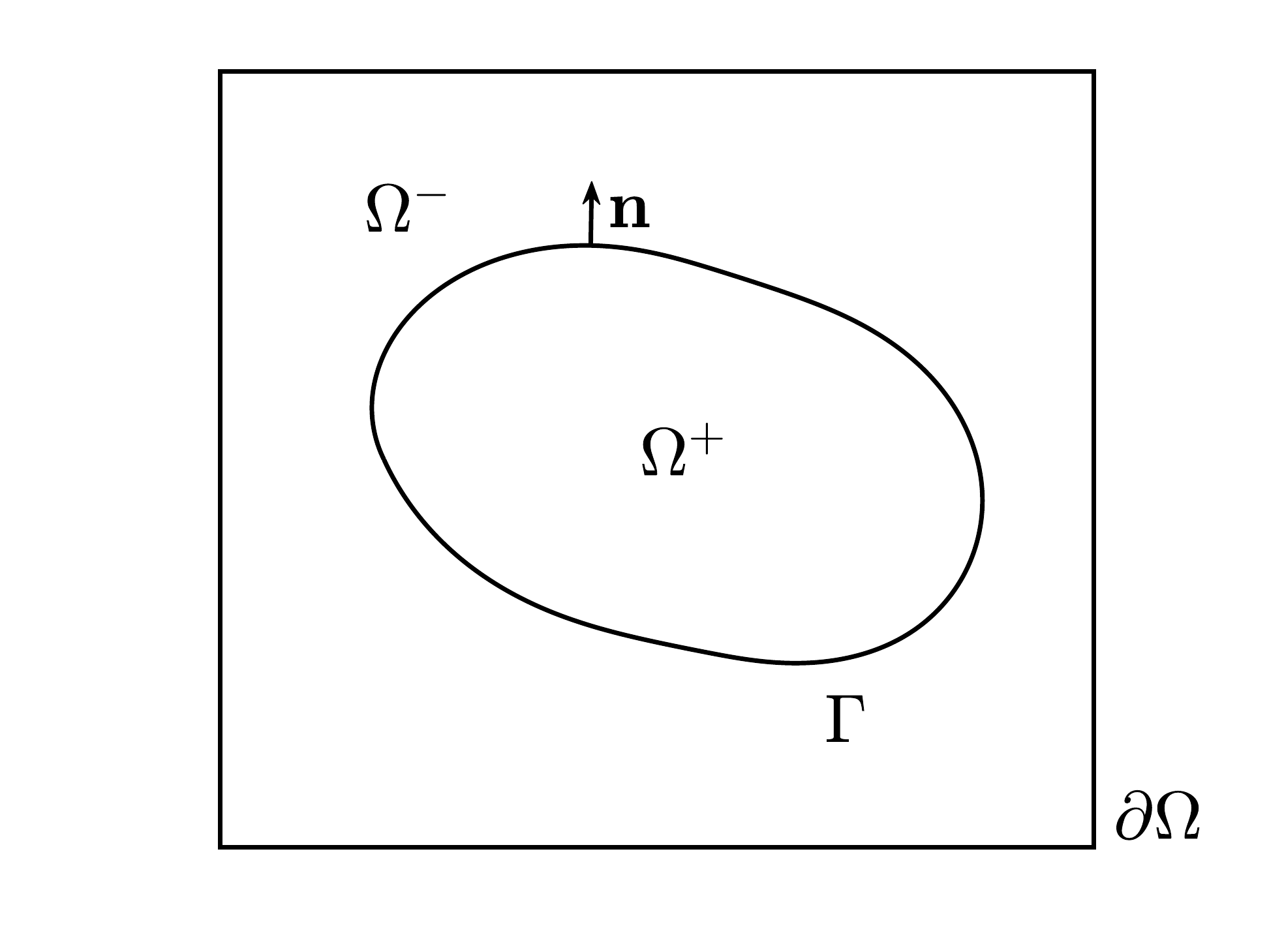}
\setlength{\abovecaptionskip}{-0.0cm}
\setlength{\belowcaptionskip}{-0.0cm}
\caption{A sketch map for the domain $\Omega$ and the interface $\Gamma$.}
\label{area}
\end{figure}

Denote the standard Sobolev space in domain $D$ by $H^k(D)$ and its 
norm by $\|\cdot\|_{H^k(D)}$. Further set the corresponding vector space $\vect   H^k(D):=[H^k(D)]^d$ and  $\|\vect   v\|_{\vect   H^k(D)}:=\sum\limits_{i=1}^d\|v_i\|_{H^k(D)}$. Let 
$L^2(D)$ be the space of all square integrable functions on $D$ and $\vect   L^2(D)$  be the corresponding vector space with inner product $(\cdot, \cdot)$. 
Define spaces 
$$\vect   V\equiv \{\vect   u\in \vect   H^1(\,\Omega),\vect   u|_{\partial \Omega}= \vect 0\}, \qquad 
M\equiv \{\,q\in L^2(\Omega):\int_{\Omega}q(\vect   x)\,d\vect   x=0\,\}.$$
The variational formulation of the interface problem \eqref{interfaceP} reads:
find $(\vect   u, p)\in \vect   V\times M$ such that 
\begin{equation}
\label{VP}
\begin{split}
&a(\vect   u, \vect   v)+ b(\vect   v, p) = (\vect   f, \vect   v)+\langle\pmb\psi, \vect   v\rangle_{\Gamma}, \;\;\forall \;\vect   v\in \vect   V,\\
&b(\vect   u, q) = 0, \;\;\forall\; q\in M,
\end{split}
\end{equation}
where  $ a(\vect   u, \vect   v) = (\nabla \vect   u, \nabla \vect   v), \;
b(\vect   v, q)=-(q,\nabla\cdot\vect   v)$, 
and
$\langle\pmb\psi, \vect   v\rangle_{\Gamma}=\int_{\Gamma}\pmb\psi\cdot\vect   vds$. 

Note that, the right hand side of equation \eqref{VP} is well-defined, thus it is well-posed \cite{hansbo2014cut}. Moreover, the following regularity for the weak solutions
$(\vect   u, p)$ of problem \eqref{VP} holds:

\begin{lemma}[\cite{shibata2003resolvent,wang2015new}]
Assume that $\vect   f\in \vect   L^2(\Omega)$ and $\pmb\psi\in \vect   H^{1/2}(\Gamma)$, then the variational problem \eqref{VP} has a unique solution $(\vect   u,p)\in \vect   V\times M$,
and the priori estimate 
\begin{equation*}
\|\vect   u\|_{\vect   H^1(\Omega)}+\|p\|_{L^2(\Omega)}\leq C(\|\vect   f\|_{\vect   L^2(\Omega)}+\|\pmb\psi\|_{\vect   H^{1/2}(\Gamma)}),
\end{equation*}
where $C$ is a generic constant independent of mesh size. 
\end{lemma}

\section{A Cartesian Grid-based MAC Scheme}

To simplify the presentation,  
assume that the computational domain is denoted by $\Omega= (0, 1)\times(0, 1)$.
Given a positive integer $N$, define
$$h= 1/N, \;\; x_i=ih,\;\; y_j=jh, \;\; 0\leq i\leq N,\;0\leq j\leq N, $$
assuming the computational domain $\Omega$ is partitioned into $N\times N$
small rectangles of the same shape. 
In the remainder, assume the given partition is fine enough to 
resolve the interface so that 
\begin{itemize}
\item[(1)]
$\Gamma$ does not intersect an edge of a rectangle at more than two points 
unless this edge is part of $\Gamma$;
\item[(2)] If $\Gamma$ meets a rectangle at two points, then these two points 
must be on two different edges of the rectangle.
\end{itemize} 

For a function $v(x, y)$, let $v_{l,m}$ denote $v(x_l, y_m)$, where $l$ may 
take values $i, i-\frac{1}{2}$ for integer $i$, and $m$ may take values 
$j, j-\frac{1}{2}$ for integer $j$.  For discrete functions, 
the discrete difference and Laplacian operators are defined by 
\begin{equation*}
\begin{split}
\delta_{h,1}^+\,v_{l,m}&=h^{-1}\left(v_{l+1,m}
-v_{l,m}\right), \;\;\quad
\delta_{h,1}^-\,v_{l,m}=h^{-1}\left(v_{l,m}
-v_{l-1,m}\right), \\[4pt]
\delta_{h,2}^+\,v_{l,m}&=h^{-1}\left(v_{l,m+1}
-v_{l,m}\right),\;\;\quad
\delta_{h,2}^-\,v_{l,m}=h^{-1}\left(v_{l,m}
-v_{l,m-1}\right), \\
\Delta_{h}v_{l,m} &= \delta_{h, 1}^+\,\delta_{h, 1}^-\,v_{l,m}
+\delta_{h, 2}^+\,\delta_{h, 2}^-\,v_{l,m}.
\end{split}
\end{equation*}

\subsection{The Maker-and-Cell Scheme}
To begin with, four different grid sets are introduced: 
a vertex-centered grid set $\mathcal{T}_h$ (the original partition), 
a cell-centered grid set $\mathcal{T}_h^0$, 
a vertical-edge-centered grid set $\mathcal{T}_h^1$, 
a horizontal-edge-centered grid set $\mathcal{T}_h^2$. 
See Fig. \ref{Fig2} for illustration. 
\begin{figure}[htbp]
\centering
\subfigure[$\mathcal{T}_h$]{
\includegraphics[width=0.22\textwidth]{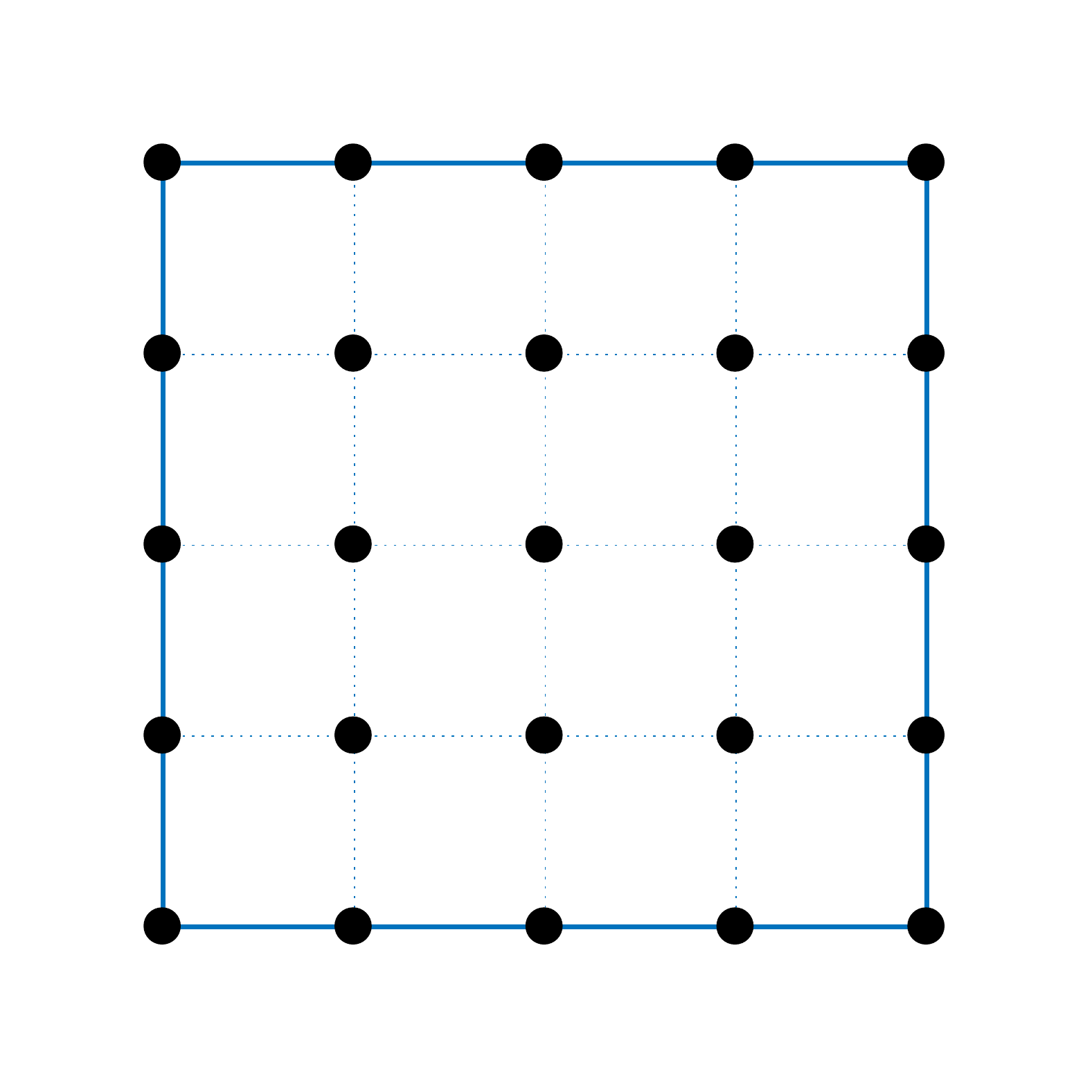}
}
\subfigure[$\mathcal{T}_h^0$]{
\includegraphics[width=0.22\textwidth]{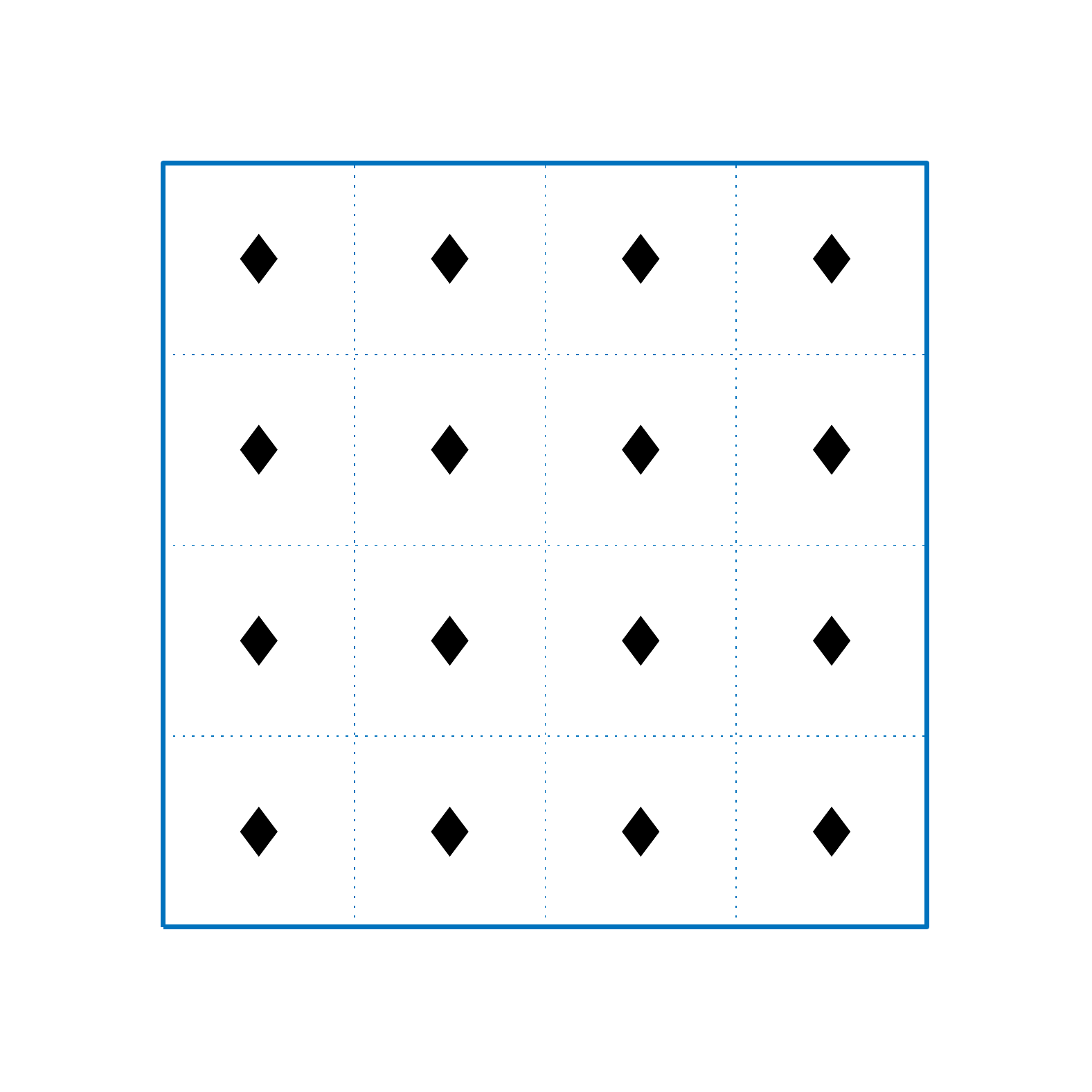}
}
\subfigure[$\mathcal{T}_h^1$]{
\includegraphics[width=0.22\textwidth]{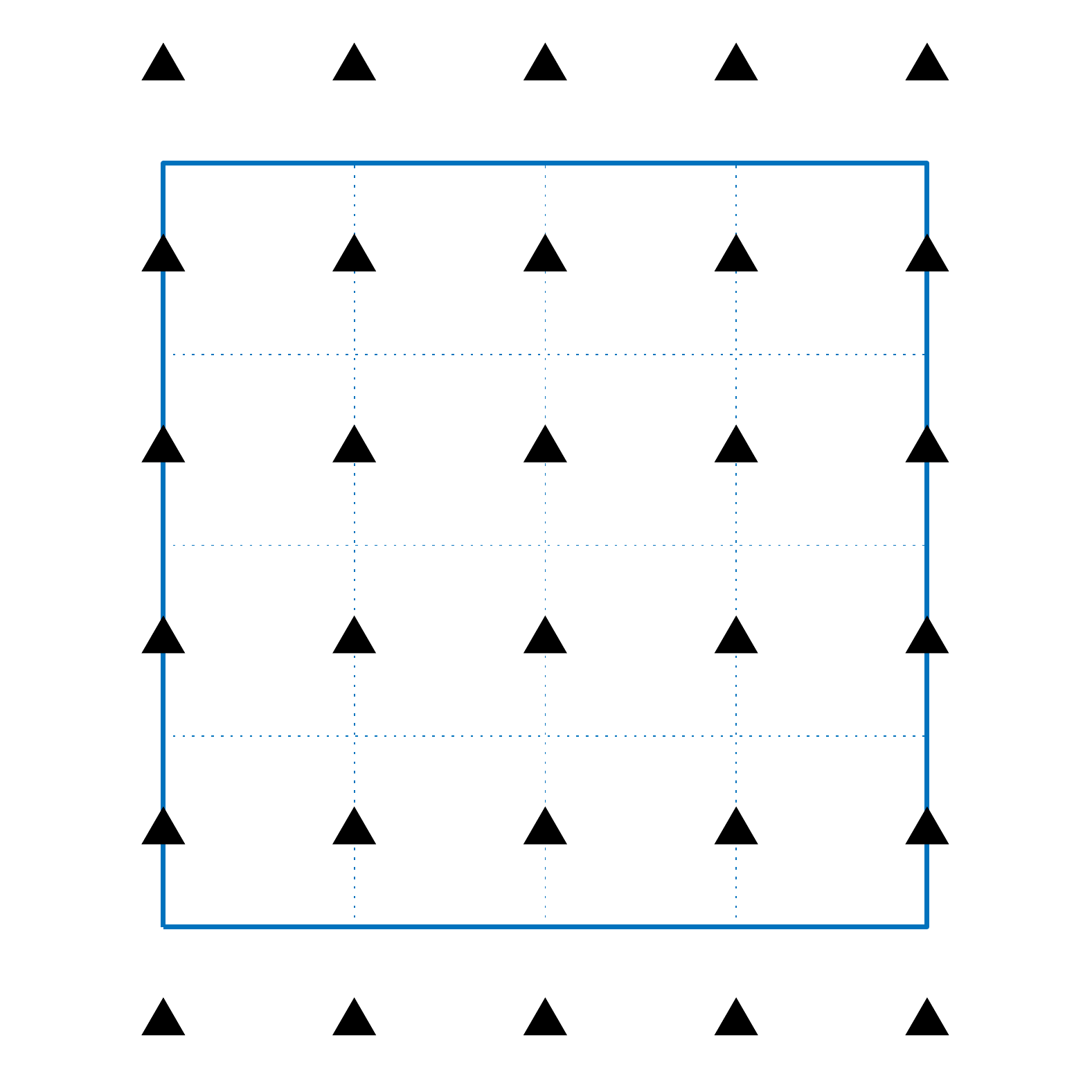}
}
\subfigure[$\mathcal{T}_h^2$]{
\includegraphics[width=0.22\textwidth]{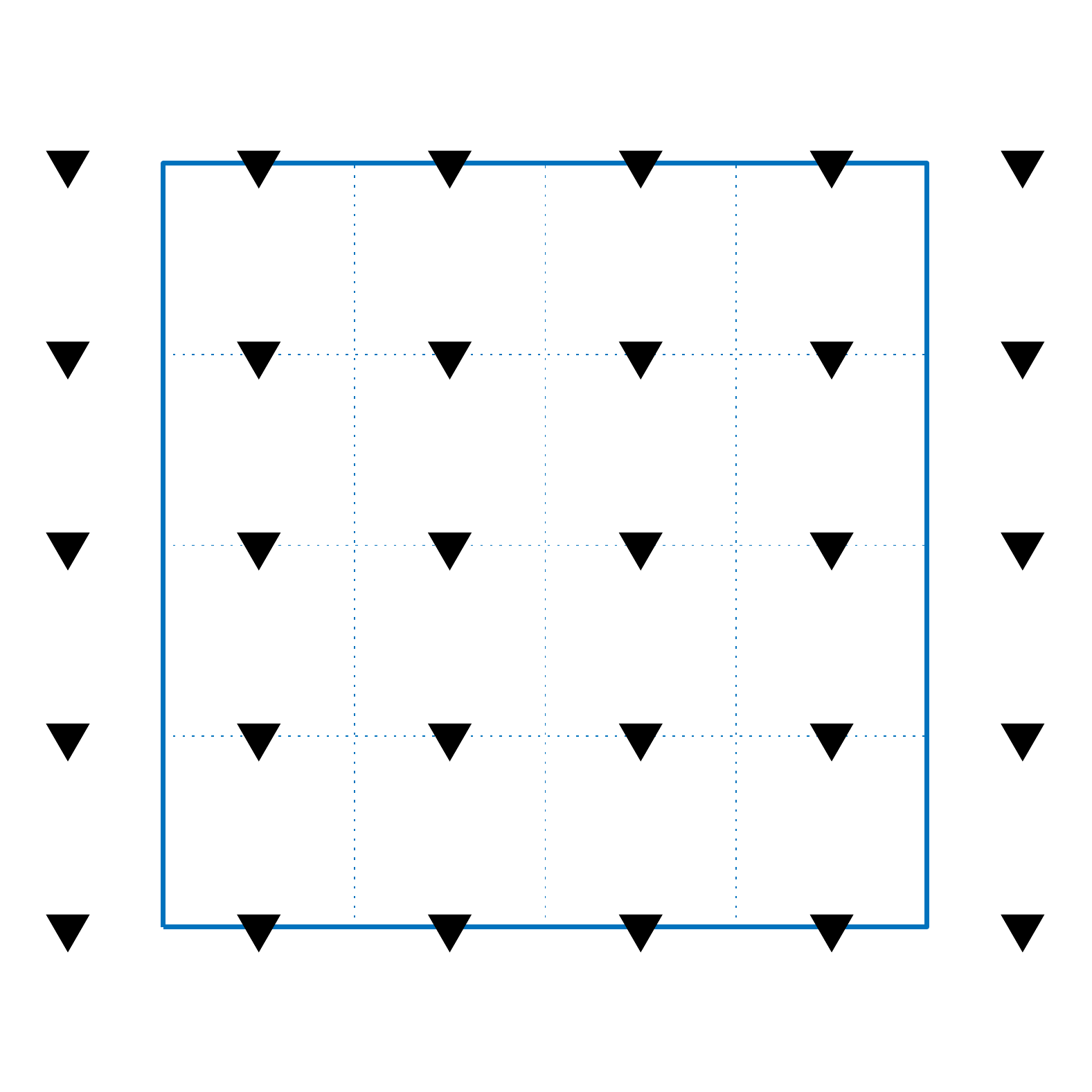}
}
\caption{Four different grid sets.}
\label{Fig2}
\end{figure}

A grid node is called regular with respect to $\Gamma$ if all grid nodes
in the corresponding finite difference stencils are on the same side of the 
interface $\Gamma$. Otherwise, it is irregular.  At a regular node, 
the MAC scheme satisfies 
\begin{equation*}
\label{regerror}
\begin{split}
-\Delta_hu_{i, j-\frac{1}{2}}^{(1)}+\delta_{h,1}^+\,p_{i-\frac{1}{2}, j-\frac{1}{2}}
=f^{(1)}_{i, j-\frac{1}{2}}, \\[4pt]
-\Delta_hu^{(2)}_{i-\frac{1}{2}, j}+\delta_{h,2}^+\,p_{i-\frac{1}{2}, j-\frac{1}{2}}
=f^{(2)}_{i-\frac{1}{2}, j}, \\[4pt]
\delta_{h,1}^-u^{(1)}_{i, j-\frac{1}{2}}+\delta_{h,2}^-u^{(2)}_{i-\frac{1}{2}, j}=0.
\end{split}
\end{equation*}
It is noted that the above MAC scheme has large local truncation errors at 
an irregular node near the interface. In order to achieve the formal second 
order accuracy, appropriate modification is needed.  By adding  some 
correction terms to the right hand side of the discrete system, 
the modified MAC scheme reads
\begin{equation}
\label{MAC}
\begin{split}
-\Delta_hu^{(1)}_{i, j-\frac{1}{2}}+\delta_{h, 1}^+\,p_{i-\frac{1}{2}, j-\frac{1}{2}}
=\tilde{f}^{(1)}_{i, j-\frac{1}{2}},&\quad\; i = 1, \cdots, N-1, \,j=1, \cdots, N,\\[4pt]
-\Delta_hu^{(2)}_{i-\frac{1}{2}, j}+\delta_{h, 2}^+\,p_{i-\frac{1}{2}, j-\frac{1}{2}}
=\tilde{f}^{(2)}_{i-\frac{1}{2}, j},&\quad\; i = 1, \cdots, N,\, j=1, \cdots, N-1,\\[4pt]
\delta_{h,1}^-\,u^{(1)}_{i, j-\frac{1}{2}}+\delta_{h,2}^-\,u^{(2)}_{i-\frac{1}{2}, j}
=\tilde{g}_{i-\frac{1}{2}, j-\frac{1}{2}},&\quad\; i = 1, \cdots, N,\, j=1, \cdots, N,
\end{split}
\end{equation}
with
\begin{equation*}
\begin{split}
&\tilde{f}^{(1)}_{i,j-\frac{1}{2}}=f^{(1)}_{i,j-\frac{1}{2}}
+C\{\Delta u^{(1)}\}_{i,j-\frac{1}{2}}+C\{p_x\}_{i,j-\frac{1}{2}},\\[4pt]
&\tilde{f}^{(2)}_{i-\frac{1}{2}, j}=f^{(2)}_{i-\frac{1}{2}, j}
+C\{\Delta u^{(2)}\}_{i-\frac{1}{2}, j}+C\{p_y\}_{i-\frac{1}{2}, j},\\[4pt]
&\tilde{g}_{i-\frac{1}{2}, j-\frac{1}{2}}=C\{u^{(1)}_x\}_{i-\frac{1}{2}, j-\frac{1}{2}}
+ C\{u^{(2)}_y\}_{i-\frac{1}{2}, j-\frac{1}{2}}.
\end{split}
\end{equation*}
Here, correction terms
\begin{equation*}
\begin{split}
&C\{\Delta u^{(1)}\}_{i, j-\frac{1}{2}}, \quad
  C\{p_x\}_{i, j-\frac{1}{2}}, \quad
  C\{u^{(1)}_x\}_{i-\frac{1}{2}, j-\frac{1}{2}},\\
&C\{\Delta u^{(2)}\}_{i-\frac{1}{2}, j},\quad
  C\{p_y\}_{i-\frac{1}{2}, j},\quad
  C\{ u^{(2)}_y\}_{i-\frac{1}{2}, j-\frac{1}{2}},
\end{split}
\end{equation*} 
are non-zero only at irregular nodes and will improve the truncation errors
near the interface to at least first order accuracy. As to be seen in 
Section \ref{sec; correction}, these correction terms can be computed 
in terms of the jumps of the solution and their derivatives. In fact, 
all jump conditions are also computable and will be derived in Section \ref{sec; jump}.

The boundary condition $u^{(1)}=0$ is imposed at the vertical physical boundary and at the ghost points which are $h/2$ to the left or right of the horizontal physical boundary. Here, the ghost points are eliminated using linear interpolation of the boundary conditions. More specifically, $$u^{(1)}_{i, -\frac{1}{2}}+u^{(1)}_{i, \frac{1}{2}}=0, \qquad u^{(1)}_{i, N-\frac{1}{2}}+u^{(1)}_{i, N+\frac{1}{2}}=0.$$ The boundary condition for the second component of the velocity is imposed similarly.  Taylor expansions on the boundary imply that the approximate boundary conditions are second 
order to the physical no-slip conditions, leading to the fact that the 
truncation errors on the boundaries are on the order of $O(1)$. However, it does not affect the global second-order accuracy, which will be illustrated in the later theoretical analysis. 

\subsection{Linear Solvers}The scheme \eqref{MAC} can be rewritten as a linear 
system in the form of
\begin{equation}
\label{linearS}
\begin{pmatrix}
-\pmb\Delta_h  & G_h^{\rm MAC}  \\[4pt]
D_h^{\rm MAC} &  0\\
\end{pmatrix}
\begin{pmatrix}
\vect   u_h\\[4pt]
p_h\\
\end{pmatrix}=
\begin{pmatrix}
\;\tilde{\!\vect   f}\;\\[4pt]
\!\tilde{g}\\
\end{pmatrix},
\end{equation}
with $\pmb\Delta_h = \hbox{diag}(\,\Delta_h, \Delta_h\,)$,  
$G_h^{\rm MAC} = (\,\delta_{h,1}^+, \delta_{h, 2}^+\,)^T$,  
$D_h^{\rm MAC}=(\,\delta_{h,1}^-, \delta_{h, 2}^-\,)$ 
and $\tilde{\!\vect   f}=(\tilde{f}^{(1)}, \tilde{f}^{(2)})^T$.
There are some fast solvers for the solution of the linear system \eqref{linearS}, 
such as the preconditioned  generalized minimal residual (GMRES)  algorithm \cite{Saad1993GMRES}, the preconditioned conjugate gradient (PCG) method  \cite{Peters2005fast}, the projection method-based pre-conditioner \cite{Griffith2009preconditioner}, the fast Fourier transform (FFT)-based method \cite{Christoph1990FFT}.  In \cite{tan2011implementation, Chen2018A}, a 
Uzawa-type method with fast solver is designed to solve this system. 
In this work, an auxiliary variable $\lambda_h$ and a parameter $\alpha$ 
are introduced to ensure uniqueness of the pressure variable $p_h$. The 
parameter $\alpha$ is chosen so that $\lambda_h$ equals the average of 
the pressure variable over the domain. The following linear system is
obtained,
 \begin{equation}
\label{linearS1}
\begin{pmatrix}
-\pmb\Delta_h  & G_h^{\rm MAC} & 0 \\[4pt]
D_h^{\rm MAC} &  0 & -\gamma \\[4pt]
0 & -\gamma^T & \alpha
\end{pmatrix}
\begin{pmatrix}
\vect   u_h\\[4pt]
p_h\\[4pt]
\lambda_h
\end{pmatrix}=
\begin{pmatrix}
\;\tilde{\!\vect   f}\;\\[4pt]
\!\tilde{g}\\[4pt]
\!0
\end{pmatrix},
\end{equation}
The linear system \eqref{linearS1} can be rewritten as
\begin{equation*}
 (D_h^{\rm MAC}\pmb\Delta_h^{-1}G_h^{\rm MAC}-\dfrac{1}{\alpha}\gamma\gamma^T)p_h=D_h^{\rm MAC}\Delta_h^{-1}\;\tilde{\!\vect   f}+\tilde{g},
\end{equation*}
which is solved with the conjugate gradient (CG) method. In this method, 
each matrix-vector product with 
$D_h^{\rm MAC}\pmb\Delta_h^{-1}G_h^{\rm MAC}$ requires solving 
two Poisson equations. In the present work, an FFT-based Poisson 
solver is employed. Once the pressure $p_h$ is obtained, the velocity 
filed $\vect   u_h$ can be derived by solving
 \begin{equation*}
 -\pmb\Delta_h \vect   u_h=\;\tilde{\!\vect   f}-G_h^{\rm MAC}p_h,
 \end{equation*}
with the FFT-based Poisson solver.

It is  remarked that the correction terms do not modify the coefficient matrix of 
the discrete system, which results from the discretization of the Stokes 
problem without an interface on a Cartesian grid. 
Thus the CG method together with the FFT-based Poisson solvers 
can be applied directly.

\subsection{Correction Terms of the MAC system}
\label{sec; correction}
As stated, since the solution is non-smooth across the interface $\Gamma$, 
the discrete equations by the MAC scheme have to be modified to avoid 
large local truncation errors at irregular grid nodes so that the global 
solution has formal second-order accuracy. In this subsection, derivation
of the correction terms used in the MAC scheme \eqref{MAC} will be 
described in the following three cases. 

\noindent{1. \em $(x_i,y_{j-\frac{1}{2}})$  is an irregular node, see 
Fig \ref{intersection} {\rm (}a{\rm )} for illustration.  }
\begin{itemize}
\item[i)] Assuming  that the interface $\Gamma$ intersects the straight line 
segment between $(x_i, y_{j-\frac{1}{2}})$ and $(x_{i+1}, y_{j-\frac{1}{2}})$ 
at point $(s_i, y_{j-\frac{1}{2}})$ with $x_i < s_i < x_{i+1}$ and 
$\xi_{u^{(1)}} = x_{i+1}-s_i$,  Taylor expansions around the intersection point 
$(s_i, y_{j-\frac{1}{2}})$ give
\begin{equation*}
C\{\Delta u^{(1)}\}_{i,j-\frac{1}{2}} = 
\dfrac{1}{h^2}\Big( [\![ u^{(1)} ]\!]  
+ \xi_{u^{(1)}} [\![ u^{(1)}_x ]\!] 
+\dfrac{1}{2}\xi_{u^{(1)}}^2 [\![ u^{(1)}_{xx} ]\!]   \Big), 
\quad \hbox{if}\; (x_i,y_{j-\frac{1}{2}})\in \Omega^+.
\end{equation*}
\item[ii)] Assuming  that the interface $\Gamma$ intersects the 
straight line segment between $(x_{i}, y_{j-\frac{1}{2}})$ and 
$(x_{i}, y_{j+\frac{1}{2}})$ at point  $(x_i,t_j)$ with 
$y_{j-\frac{1}{2}} < t_j < y_{j+\frac{1}{2}}$ and 
$\eta_{u^{(1)}} = y_{j+\frac{1}{2}}-t_j$,  Taylor expansions around 
the intersection point $(x_{i}, t_{j})$ give
\begin{equation*}
C\{\Delta u^{(1)}\}_{i,j-\frac{1}{2}} = 
\dfrac{1}{h^2}\Big( [\![ u^{(1)} ]\!]  
+ \eta_{u^{(1)}} [\![ u^{(1)}_y ]\!] 
+\dfrac{1}{2}\eta_{u^{(1)}}^2 [\![ u^{(1)}_{yy} ]\!]   \Big), 
\quad \hbox{if}\; (x_{i},y_{j-\frac{1}{2}})\in \Omega^+.
\end{equation*}
\end{itemize}
\begin{figure}[ht!]
\centering
\subfigure[$irregular (x_{i}, y_{j-\frac{1}{2}})$]{
\includegraphics[width=0.22\textwidth]{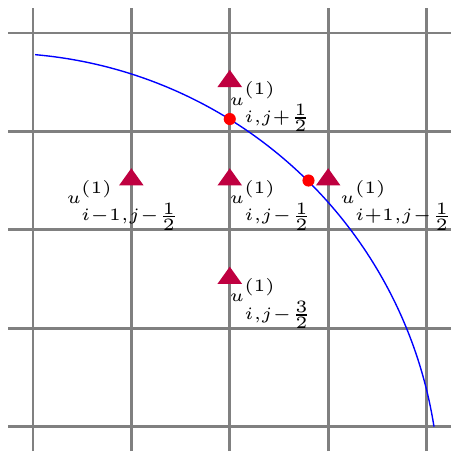}
}
\subfigure[$irregular (x_{i-\frac{1}{2}}, y_{j})$]{
\includegraphics[width=0.22\textwidth]{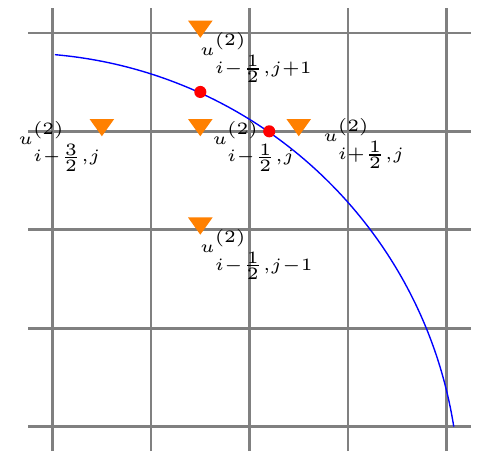}
}
\subfigure[$irregular (x_{i-\frac{1}{2}}, y_{j-\frac{1}{2}})$]{
\includegraphics[width=0.22\textwidth]{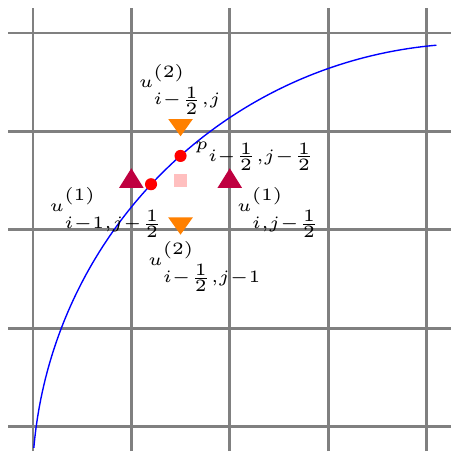}
}
\subfigure[$irregular (x_{i-\frac{1}{2}}, y_{j-\frac{1}{2}})$]{
\includegraphics[width=0.22\textwidth]{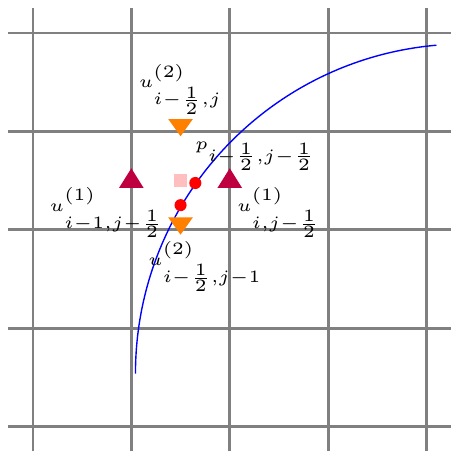}
}
\setlength{\abovecaptionskip}{-0.0cm}
\setlength{\belowcaptionskip}{-0.4cm} 
\caption{A diagram of the interface cutting through a staggered grid 
around an irregular point}
\label{intersection}
\end{figure}
\noindent{2. \em $(x_{i-\frac{1}{2}},y_{j})$  is an irregular node, 
see Fig \ref{intersection} {\rm (}b{\rm )} for illustration.  }
\begin{itemize}
\item[i)] Assuming  that the interface $\Gamma$ intersects the straight line 
segment between $(x_{i-\frac{1}{2}}, y_j)$ and $(x_{i+\frac{1}{2}}, y_{j})$ 
at point  $(s_i, y_{j})$ with $x_{i-\frac{1}{2}} < s_i < x_{i+\frac{1}{2}}$ 
and $\xi_{u^{(2)}} = x_{i+\frac{1}{2}}-s_i$,  Taylor expansions around 
the intersection point $(s_i, y_{j})$ give 
\begin{equation*}
C\{\Delta u^{(2)}\}_{i-\frac{1}{2},j} = 
\dfrac{1}{h^2}\Big( [\![ u^{(2)} ]\!]  
+ \xi_{u^{(2)}} [\![ u^{(2)}_x ]\!] 
+\dfrac{1}{2}\xi_{u^{(2)}}^2 [\![ u^{(2)}_{xx} ]\!]   \Big), 
\quad \hbox{if}\; (x_{i-\frac{1}{2}},y_j)\in \Omega^+.
\end{equation*}
\item[ii)] Assuming  that the interface $\Gamma$ intersects the straight 
line segment between  $(x_{i-\frac{1}{2}}, y_{j})$ and 
$(x_{i-\frac{1}{2}}, y_{j+1})$ at point $(x_{i-\frac{1}{2}},t_j)$ 
with $y_{j} < t_j < y_{j+1}$ and $\eta_{u^{(2)}} = y_{j+1}-t_j$,  
Taylor expansions around the intersection point $(x_{i-\frac{1}{2}}, t_{j})$ give
\begin{equation*}
C\{\Delta u^{(2)}\}_{i-\frac{1}{2},j} = 
\dfrac{1}{h^2}\Big( [\![ u^{(2)} ]\!]  
+ \eta_{u^{(2)}} [\![ u^{(2)}_y ]\!] 
+\dfrac{1}{2}\eta_{u^{(2)}}^2 [\![ u^{(2)}_{yy} ]\!]   \Big), 
\quad \hbox{if}\; (x_{i-\frac{1}{2}},y_{j})\in \Omega^+.
\end{equation*}
\end{itemize}
\noindent{3. \em $(x_{i-\frac{1}{2}},y_{j-\frac{1}{2}})$  is an irregular node, 
see Fig \ref{intersection} {\rm (}c{\rm )}-{\rm (}d{\rm )} for illustration.  }
\begin{itemize}
\item[i)] Assuming  that the interface $\Gamma$ intersects the straight 
line segment between $(x_{i-1}, y_{j-\frac{1}{2}})$ and 
$(x_{i-\frac{1}{2}}, y_{j-\frac{1}{2}})$ at point $(s_i, y_{j-\frac{1}{2}})$ 
with $x_{i-1} < s_i < x_{i-\frac{1}{2}}$ and 
$\xi_{u^{(1)}} = x_{i-1}-s_i,\; \xi_p=x_{i-\frac{1}{2}}-s_i$,  Taylor expansions 
around the intersection point $(s_i, y_{j-\frac{1}{2}})$ give 
\begin{equation*}
C\{u_x^{(1)}\}_{i-\frac{1}{2},j-\frac{1}{2}} = 
\dfrac{1}{h}\Big( [\![ u^{(1)} ]\!]  
+ \xi_{u^{(1)}} [\![ u^{(1)}_x ]\!] 
+\dfrac{1}{2}\xi_{u^{(1)}}^2 [\![ u^{(1)}_{xx} ]\!]   \Big), 
\quad \hbox{if}\; (x_{i-\frac{1}{2}},y_{j-\frac{1}{2}})\in \Omega^+,
\end{equation*}
and 
\begin{equation*}
C\{p_x\}_{i-1,j-\frac{1}{2}}=
-\dfrac{1}{h}\Big( [\![ p ]\!]  
+ \xi_{p} [\![ p_x ]\!]  \Big), 
\quad \hbox{if}\; (x_{i-1},y_{j-\frac{1}{2}})\in \Omega^+.
\end{equation*}
\item[ii)] Assuming  that the interface $\Gamma$ intersects the straight 
line segment between $(x_{i-\frac{1}{2}}, y_{j-\frac{1}{2}})$ and 
$(x_{i}, y_{j-\frac{1}{2}})$ at point  $(s_i, y_{j-\frac{1}{2}})$ with 
$x_{i-\frac{1}{2}} < s_i < x_{i}$ 
and $\xi_{u^{(1)}} = x_{i}-s_i,\; \xi_p=x_{i-\frac{1}{2}}-s_i$,  Taylor expansions 
around the intersection point $(s_i, y_{j-\frac{1}{2}})$ give 
\begin{equation*}
 C\{u_x^{(1)}\}_{i-\frac{1}{2},j-\frac{1}{2}} = 
- \dfrac{1}{h}\Big( [\![ u^{(1)} ]\!]  
+ \xi_{u^{(1)}} [\![ u^{(1)}_x ]\!] 
+\dfrac{1}{2}\xi_{u^{(1)}}^2 [\![ u^{(1)}_{xx} ]\!]   \Big), 
\quad \hbox{if}\; (x_{i-\frac{1}{2}},y_{j-\frac{1}{2}})\in \Omega^+,
\end{equation*}
and 
\begin{equation*}
C\{p_x\}_{i,j-\frac{1}{2}}=
\dfrac{1}{h}\Big( [\![ p ]\!]  
+ \xi_{p} [\![ p_x ]\!]  \Big), 
\quad \hbox{if}\; (x_{i},y_{j-\frac{1}{2}})\in \Omega^+.
\end{equation*}
\item[iii)] Assuming that the interface $\Gamma$ intersects the straight
 line segment between $(x_{i-\frac{1}{2}}, y_{j-1})$ and 
 $(x_{i-\frac{1}{2}}, y_{j-\frac{1}{2}})$ at point  $(x_{i-\frac{1}{2}}, t_{j})$ 
 with $y_{j-1} < t_j < y_{j-\frac{1}{2}}$ and 
 $\eta_{u^{(2)}} = y_{j-1}-t_j,\; \eta_p=y_{i-\frac{1}{2}} - t_j$,  Taylor expansions 
 around the intersection point $(x_{i-\frac{1}{2}}, t_j)$ give 
\begin{equation*}
C\{u_y^{(2)}\}_{i-\frac{1}{2},j-\frac{1}{2}} = 
\dfrac{1}{h}\Big( [\![ u^{(2)} ]\!]  
+ \eta_{u^{(2)}} [\![ u^{(2)}_y ]\!] 
+\dfrac{1}{2}\eta_{u^{(2)}}^2 [\![ u^{(2)}_{yy} ]\!]   \Big), 
\quad \hbox{if}\; (x_{i-\frac{1}{2}},y_{j-\frac{1}{2}})\in \Omega^+,
\end{equation*}
and
\begin{equation*}
C\{p_y\}_{i-\frac{1}{2},j-1}=
-\dfrac{1}{h}\Big( [\![ p ]\!]  
+ \eta_{p} [\![ p_y ]\!]  \Big), 
\quad \hbox{if}\; (x_{i-\frac{1}{2}},y_{j-1})\in \Omega^+.
\end{equation*}
\item[iv)] Assuming that the interface $\Gamma$ intersects the straight 
line segment between $(x_{i-\frac{1}{2}}, y_{j-\frac{1}{2}})$ and 
$(x_{i-\frac{1}{2}}, y_{j})$ at point $(x_{i-\frac{1}{2}}, t_{j})$ 
with $y_{j-\frac{1}{2}} < t_j < y_{j}$ and 
$\eta_{u^{(2)}} = y_{j}-t_j,\; \eta_p=y_{i-\frac{1}{2}} - t_j$,  Taylor expansions 
around the intersection point $(x_{i-\frac{1}{2}}, t_j)$ give
\begin{equation*}
C\{u_y^{(2)}\}_{i-\frac{1}{2},j-\frac{1}{2}} = 
-\dfrac{1}{h}\Big( [\![ u^{(2)} ]\!]  
+ \eta_{u^{(2)}} [\![ u^{(2)}_y ]\!] 
+\dfrac{1}{2}\eta_{u^{(2)}}^2 [\![ u^{(2)}_{yy} ]\!]   \Big),
 \quad \hbox{if}\; (x_{i-\frac{1}{2}},y_{j-\frac{1}{2}})\in \Omega^+,
\end{equation*}
and
\begin{equation*}
C\{p_y\}_{i-\frac{1}{2},j}=
\dfrac{1}{h}\Big( [\![ p ]\!]  
+ \eta_{p} [\![ p_y ]\!]  \Big), 
\quad \hbox{if}\; (x_{i-\frac{1}{2}},y_{j})\in \Omega^+.
\end{equation*}
\end{itemize}

Correction terms at an irregular grid node 
$(x_i,y_{j-\frac{1}{2}}),\,(x_{i-\frac{1}{2}},y_j)$ or 
$(x_{i-\frac{1}{2}},y_{j-\frac{1}{2}})$ located in the domain $\Omega^-$ 
can be obtained similarly. Actually, it is completely the same as that for 
irregular grid nodes in the domain $\Omega^+$ except each correction 
term should be negated. 
It is worth pointing out that derivation of the correction terms indicates 
the local truncation errors of the modified MAC scheme \eqref{MAC}  
at irregular points are first order for the 
first two equations and second order for the third equation. The later
 theoretical analysis shows that this is sufficient to guarantee the global 
 second-order accuracy.

Once again, the jumps of partial derivatives of the solution to the interface, 
which are involved in the correction terms, will be computed 
in Section \ref{sec; jump}. 

\subsection{Calculation for Jump Conditions}
\label{sec; jump}
This section describes the calculation for the jumps of partial derivatives 
of $u^{(1)}, u^{(2)},$ and $p$ respectively, 
which will be uniquely determined by the given jump conditions 
$[\![ \vect   u ]\!]$ and $[\![\pmb \sigma\vect   n]\!]$ 
in \eqref{interfaceP}. 

For simplicity, $x'$ and $y'$ are respectively used to denote $dx/d\eta$ 
and $dy/d\eta$,  $x''$ and $y''$ are respectively used to denote 
$d^2x/d\eta^2$ and $d^2y/d\eta^2$, where $\eta$ represents the 
tangential direction. Differentiating $[\![ \vect u ]\!] = \vect 0$ in \eqref{interfaceP} 
 with respect to the tangential direction $\eta$ gives
\begin{equation}
\label{jumplocal1}
[\![ u^{(1)}_{x} ]\!] x'+[\![ u^{(1)}_y ]\!] y'=0,\qquad
[\![ u^{(2)}_{x} ]\!] x'+[\![ u^{(2)}_y ]\!] y'=0.
\end{equation}
Moreover, equation $[\![ \pmb \sigma(\vect   u, p)\vect   n]\!]=\pmb \psi$ explicitly reads
\begin{subequations}
\label{jumplocal2}
\begin{align}
2[\![ u^{(1)}_x]\!] n_1+\big([\![ u^{(1)}_y]\!] +
[\![ u^{(2)}_x]\!]\big)n_2-[\![ p]\!] n_1&=\psi_1,\\[4pt]
\big([\![ u^{(2)}_x]\!] +
[\![ u^{(1)}_y]\!]\big)n_1+2[\![ u^{(2)}_y]\!] n_2
-[\![ p]\!] n_2&=\psi_2.
\end{align}
\end{subequations}
Therefore, equations \eqref{jumplocal1}-\eqref{jumplocal2}
 together with  
 \begin{equation}
 \label{jumplocal3}
 [\![ u^{(1)}_x ]\!] +[\![ u^{(2)}_y ]\!] =0,   
 \end{equation}
form a $5$ by $5$ linear system, solving which yields the jumps of 
the first order  partial derivatives of the velocity $\vect   u$ and the 
jump of the pressure $p$.
Differentiating the equation \eqref{jumplocal3} along the $x$- and $y$- directions respectively, and taking tangential derivative of \eqref{jumplocal1}-\eqref{jumplocal2}, together with the first equations of \eqref{interfaceP}, an 8 by 8 linear system is obtained, which reads 
\begin{equation*}
\begin{pmatrix}
1  &  0  &  0   &   0  &   1    &0     &  0  &   0  \\[4pt]
0  & 1  &  0   &   0  &   0   & 1   & 0   &  0  \\[4pt] 
 -1  & 0    &   -1    &  0 &   0   &  0   &  1  &  0 \\[4pt]
  0 &   0   &  0   &-1  & 0    &   -1    &    0  &  1 \\[4pt]
(x')^2 &  2x'y' &  (y')^2  & 0 &0 &0 &0 &0\\[4pt]
0 & 0 & 0 & (x')^2 & 2x'y' &  (y')^2 & 0 &0\\[4pt]
2n_1x' &  2n_1y'+n_2x'  & n_2y' & n_2x' & n_2y' & 0 &-n_1x'   &-n_1y'\\[4pt]
0 &  n_1x'& n_1y' &  n_1x' & n_1y'+2n_2x' & 2n_2y' & -n_2x'  &-n_2y'
 \end{pmatrix}
 \begin{pmatrix}
[\![  u^{(1)}_{xx} ]\!]\\[4pt]
[\![  u^{(1)}_{xy} ]\!]\\[4pt]
[\![  u^{(1)}_{yy} ]\!]\\[4pt]
[\![  u^{(2)}_{xx} ]\!]\\[4pt]
[\![  u^{(2)}_{xy} ]\!]\\[4pt]
[\![  u^{(2)}_{yy} ]\!]\\[4pt]
[\![  p_{x} ]\!]\\[4pt]
[\![  p_{y} ]\!]
 \end{pmatrix}=
  \begin{pmatrix}
r_1\\[4pt]
r_2\\[4pt]
r_3\\[4pt]
r_4\\[4pt]
r_5\\[4pt]
r_6\\[4pt]
r_7\\[4pt]
r_8
 \end{pmatrix}
 \end{equation*}
with 
\begin{equation*}
\begin{split}
r_1&=r_2=0,\quad r_3=[\![ f^{(1)}]\!],\quad 
r_4=[\![ f^{(2)}]\!], \\[4pt]
r_5&=-[\![  u^{(1)}_x]\!] x''-[\![  u^{(1)}_y]\!] y'',\quad
r_6=-[\![  u^{(2)}_x]\!] x''-[\![  u^{(2)}_y]\!] y'',\\[4pt]
r_7&=\psi'_1-2[\![ u^{(1)}_x]\!] n'_1
-\big([\![ u^{(1)}_y]\!] +
[\![ u^{(2)}_x]\!]\big)n'_2+[\![ p]\!] n'_1, \\[4pt]
r_8&=\psi'_2-\big([\![ u^{(2)}_x]\!] +
[\![ u^{(1)}_y]\!]\big)n'_1-2[\![ u^{(2)}_y]\!] n'_2
+[\![ p]\!]  n'_2.
\end{split}
\end{equation*} 
From these eight equations, one can get the jumps of the second-order 
partial derivatives $[\![ u^{(1)}_{xx}]\!]$, 
$[\![ u^{(1)}_{xy}]\!]$,
$[\![ u^{(1)}_{yy}]\!]$, $[\![ u^{(2)}_{xx}]\!]$,
$[\![ u^{(2)}_{xy}]\!]$,$[\![ u^{(2)}_{yy}]\!]$,
and the jumps of the first-order partial  derivatives 
$[\![ p_x]\!], [\![ p_y]\!]$.


\section{$\ell^2$ Error Analysis}
\label{sec; truncation}
In this section, the detailed discussion of the $\ell^2$-error analysis 
for the MAC scheme \eqref{MAC} is given.  For this purpose, some notations
are first introduced. Denote the following different grid function spaces:
\begin{eqnarray*}
&&V_h^{(1)}=\Big\{\,v^{(1)}_{i,j-\frac{1}{2}},\;\;
i=0,\cdots,N,\; \;j=0,\cdots,N+1, \\
&& \quad\qquad\qquad\qquad  
v^{(1)}_{0,j-\frac{1}{2}}=v^{(1)}_{N,j-\frac{1}{2}}=0,\;\;
v^{(1)}_{i,-\frac{1}{2}}=-v^{(1)}_{i,\frac{1}{2}},\;\;
v^{(1)}_{i,N+\frac{1}{2}}=-v^{(1)}_{i,N-\frac{1}{2}}\,\Big\},\\[4pt] 
&&V_h^{(2)}=\Big\{\,v^{(2)}_{i-\frac{1}{2},j}, \;\;
i=0,\cdots,N+1,\;\; j=0,\cdots,N, \\
&& \quad\qquad\qquad\qquad 
 v^{(2)}_{i-\frac{1}{2},0}=v^{(2)}_{i-\frac{1}{2},N}=0, \;\;
v^{(2)}_{-\frac{1}{2},j}=-v^{(2)}_{\frac{1}{2},j},\;\;
v^{(2)}_{N+\frac{1}{2},j}=-v^{(2)}_{N-\frac{1}{2},j}\,\Big\},\\[4pt]
&&M_h=\Big\{\,q_{i-\frac{1}{2},j-\frac{1}{2}},\;\;
i=1,\cdots,N,\;\;j=1,\cdots,N, \quad \;\; 
\sum\limits_{i=1}^N\sum\limits_{i=1}^N 
q_{i-\frac{1}{2}, j-\frac{1}{2}}=0 \,\Big\},\\[4pt]
&&W_h^{(1)}=\Big\{\,w_{i,j}, \;\; 
i=0,\cdots,N, \;\;j=0,\cdots,N,
\;\;\quad w_{0,j}=w_{N, j}=0 \,\Big\},\\[4pt]
&&W_h^{(2)}=\Big\{\,w_{i,j},\;\;
 i=0,\cdots,N, \;\;j=0,\cdots,N, 
 \;\;\quad w_{i,0}=w_{i,N}=0\,\Big\},\\[4pt]
&&\vect   V_h = V_h^{(1)}\times V_h^{(2)}.
\end{eqnarray*}
For grid functions $v_h$ and $w_h$, define the discrete $\ell^2$-inner 
products as follows:
\begin{equation*}
\begin{split}
(v_h, w_h)_{V_h^{(1)}}&\equiv h^2\sum\limits_{i=1}^{N-1} 
\sum\limits_{j=1}^{N}v_{i, j-\frac{1}{2}} w_{i,j-\frac{1}{2}},\;\;\quad\;\,
(v_h,  w_h)_{V_h^{(2)}}\equiv h^2\sum\limits_{i=1}^{N} 
\sum\limits_{j=1}^{N-1}v_{i-\frac{1}{2},j} w_{i-\frac{1}{2}, j},\\[4pt]
(v_h,  w_h)_{W_h^{(1)}}& \equiv h^2\sum_{i=1}^{N-1}
\sum_{j=0}^{N}\rho_{j}^y v_{i,j} w_{i,j},\qquad\quad
(v_h,  w_h)_{W_h^{(2)}} \equiv h^2\sum_{i=0}^{N}
\sum_{j=1}^{N-1}\rho_{i}^x v_{i,j} w_{i,j},\\[4pt]
(v_h,  w_h)_{M_h}&\equiv h^2 \sum\limits_{i=1}^{N} 
\sum\limits_{j=1}^{N}v_{i-\frac{1}{2}, j-\frac{1}{2}} w_{i-\frac{1}{2},j-\frac{1}{2}},
\end{split}
\end{equation*}
where $\rho_{0}^x=\rho_{N}^x=\dfrac{1}{2}, \rho_{i}^x=1$ when 
$i=1, 2, \cdots, N-1$, and  $\rho_{0}^y=\rho_{N}^y=\dfrac{1}{2},\rho_{j}^y=1$ 
when $j=1, 2, \cdots, N-1$.  
The corresponding discrete $\ell^2$-norms are denoted as
\begin{equation*}
\|v_h\|^2_{S} \equiv (v_h, v_h)_{S}, \; S = V_h^{(1)}, \,V_h^{(2)},\, W_h^{(1)},\, 
W_h^{(2)}\hbox{or}\; M_h.
\end{equation*}
For a vector-valued function $\vect   v_h=(v_h^{(1)}, v_h^{(2)})$ with 
$v_h^{(1)}\in V_h^{(1)}$ and $v_h^{(2)}\in V_h^{(2)}$, define
\begin{equation*}
\begin{split}
&\|\vect   v_h\|^2\equiv \|v_h^{(1)}\|_{V_h^{(1)}}^2+ \|v_h^{(2)}\|_{V_h^{(2)}}^2,\\[4pt]
&|\vect   v_h|_1^2\;\equiv \|\delta_{h,1}^-\;v_h^{(1)}\|_{M_h}^2+\|\delta_{h,2}^-\,v_h^{(1)}\|_{W_h^{(1)}}^2+\|\delta_{h, 1}^-\,v_h^{(2)}\|_{W_h^{(2)}}^2+\|\delta_{h,2}^-\,v_h^{(2)}\|_{M_h}^2.
\end{split}
\end{equation*}
Moreover, denote the maximum norm for the $r$-th derivatives of 
any function $v$ as 
$$\|v\|_{r, \infty}=\max\left|\partial^{s+l}v/\partial x^s\partial y^l\right|,$$
where $s+l\leq r$, and $s, l\geq 0$. 

Analogous to the continuous cases, there are the discrete 
version for Green's formulae and Poincare inequalities.
For details of the proofs, one can refer to \cite{dong2020Maximum}.
\begin{lemma}
\label{dGf}
{\em For $v_h^{(1)}, \widetilde{v}_h^{(1)}\in V_h^{(1)}, v_h^{(2)},\widetilde{v}_h^{(2)}
\in V_h^{(2)}, p_h\in M_h$,  the following discrete Green's formulae hold:
\begin{eqnarray*}
&&(\delta_{h, 1}^+\,p_h, v_h^{(1)})_{V_h^{(1)}}
=-(p_h, \delta_{h,1}^-\,v_h^{(1)})_{M_h},\\[4pt]
&&(\delta_{h, 2}^+\,p_h, v_h^{(2)})_{V_h^{(2)}}
=-(p_h, \delta_{h, 2}^-\,v_h^{(2)})_{M_h},\\[4pt]
&&(-\Delta_{h}v_h^{(1)},\widetilde{v}_h^{(1)})_{V_h^{(1)}}
=(\delta_{h,1}^-\,v_h^{(1)},\delta_{h, 1}^-\,\widetilde{v}_h^{(1)})_{M_h}+
(\delta_{h, 2}^-\,v_h^{(1)},\delta^-_{h, 2}\widetilde{v}_h^{(1)})_{W_h^{(1)}},\\[4pt]
&&(-\Delta_{h}v_h^{(2)},\widetilde{v}_h^{(2)})_{V_h^{(2)}}
=(\delta_{h, 1}^-\,v_h^{(2)},\delta^-_{h, 1}\,\widetilde{v}_h^{(2)})_{W_h^{(2)}}+
(\delta_{h, 2}^-\,v_h^{(2)},\delta_{h, 2}^-\,\widetilde{v}_h^{(2)})_{M_h}.
\end{eqnarray*}
}
\end{lemma}
\begin{lemma}
{\em Under the assumption that $v_h^{(1)}\in V_h^{(1)}, v_h^{(2)}\in V_h^{(2)}$, 
it holds that 
\label{disPoin}
\begin{equation*}
\begin{split}
\|v_h^{(1)}\|^2_{V_h^{(1)}}&\leq C_1(\|\delta_{h,1}^-\,v_h^{(1)}\|_{M_h}^2
+\|\delta_{h,2}^-\,v_h^{(2)}\|_{W_h^{(1)}}^2),\\[4pt]
\|v_h^{(2)}\|^2_{V_h^{(2)}}&\leq C_2(\|\delta_{h,1}^-\,v_h^{(2)}\|_{W_h^{(2)}}^2
+\|\delta_{h,2}^-\,v_h^{(2)}\|_{M_h}^2),
\end{split}
\end{equation*}
where the constants $C_1$  and $C_2$ only depend on the area of the 
computational domain.}
\end{lemma}

An important part of the theoretical analysis is the discrete LBB condition 
for the MAC discretization, which was first proven by Shin and Strickwerda 
on uniform meshes in \cite{shin1997inf} and  by Blanc on non-uniform 
meshes in \cite{blanc1999error,blanc2005convergence}.  Later, 
the results were improved in the work by Gallou\" et et al. \cite{Herbin2012W1}.The discrete LBB condition is stated in the following lemma.
\begin{lemma}
\label{LBB}
{\rm There exists a constant $\beta>0$ independent of  the mesh parameter 
$h$ such that
\begin{equation}
\sup\limits_{\vect   v_h\in\vect   V_h}\frac{b_h(\vect   v_h, q_h)}{|\vect   v_h|_1}\geq 
\beta\|q_h\|_{M_h}, \;\;\forall q_h\in M_h,
\end{equation}
where
\begin{equation*}
b_h(\vect   v_h, q_h)=-h^2\sum\limits_{i=1}^N\sum\limits_{j=1}^N
q_{i-\frac{1}{2}, j-\frac{1}{2}}\nabla_h\cdot\vect   v_{i,j}, 
\;\vect   v_h\in \vect   V_h,\;q_h\in M_h,
\end{equation*}
with 
$\nabla_h\cdot\vect   v_{i,j}=\delta_{h,1}^-\,v^{(1)}_{i, j-\frac{1}{2}}
+\delta_{h,2}^-\,v^{(2)}_{i-\frac{1}{2}, j}.$
}
\end{lemma}

With the LBB condition in hand, the convergence of numerical solutions of 
the Stokes problem will be considered later. As explained before, truncation 
errors at the internal regular points are of second order, at the internal 
irregular points are of first order and are only $\mathcal{O}(1)$ near the 
boundary points.  In order to obtain the formal second order convergence, 
an important ingredient is the construction of approximate solutions that 
satisfy the discrete equations \eqref{MAC} to  high order accuracy. The 
following lemma states how to construct the desired auxiliary functions.

\begin{lemma}
\label{consistence}
{\rm Assume the solutions of the Stokes  interface equations  are sufficiently 
smooth in $\Omega$ excluding $\Gamma$. There exist functions 
$\tilde{u}^{(1)}, \tilde{u}^{(2)}$ and $\tilde{p}$, which are $\mathcal{O}(h^2)$ 
perturbations of the exact solutions $u^{(1)}, u^{(2)}$ and $p$, to satisfy 
\begin{equation}
\label{diuvp}
\begin{split}
&\widetilde{u}^{(1)}(x, y, h) = u^{(1)}(x, y)+h^2\hat{u}^{(1)}(x, y),\\[4pt]
&\widetilde{u}^{(2)}(x, y, h) = u^{(2)}(x, y)+h^2\hat{u}^{(2)}(x, y),\\[4pt]
&\widetilde{p}(x, y, h) = p(x, y)+h^2\hat{p}(x, y),
\end{split}
\end{equation} 
where the functions $\hat{u}^{(1)}, \hat{u}^{(2)}$ and $\hat{p}$ and  their 
derivatives can be bounded in terms of the analytic solutions.  
Let $\widetilde{u}^{(1)}_{i, j-\frac{1}{2}}= \widetilde{u}(x_i, y_{j-\frac{1}{2}},h)$. 
In the same manner,  $\widetilde{u}^{(2)}_{i-\frac{1}{2},j}$,\, 
$ \widetilde{p}_{i-\frac{1}{2}, j-\frac{1}{2}}$,
\,$\hat{u}^{(1)}_{i, j-\frac{1}{2}}$,
\, $\hat{u}^{(2)}_{i-\frac{1}{2},j}$,\, 
$\hat{p}_{i-\frac{1}{2}, j-\frac{1}{2}}$ are defined. 
These functions satisfy the discrete equations to higher order accuracy 
in the following sense:
\begin{subequations}
\label{high}
\begin{align}
&-\Delta_h\widetilde{u}_{i, j-\frac{1}{2}}^{(1)}+\delta_{h,1}^+
\,\widetilde{p}_{i-\frac{1}{2}, j-\frac{1}{2}}=\tilde{f}^{(1)}_{i, j-\frac{1}{2}}
+\widetilde{R}^{(1)}_{i, j-\frac{1}{2}}, \;\;\; \,{\rm in }\; V_h^{(1)}\label{h1},\\[4pt]
&-\Delta_h\widetilde{u}^{(2)}_{i-\frac{1}{2}, j}+\delta_{h,2}^+
\,\widetilde{p}_{i-\frac{1}{2}, j-\frac{1}{2}}=\tilde{f}^{(2)}_{i-\frac{1}{2}, j}
+\widetilde{R}^{(2)}_{i-\frac{1}{2}, j}, \; \;\;\,\hbox{\rm in }\; 
V_h^{(2)},\label{h2}\\[4pt]
&\delta_{h,1}^-\,\widetilde{u}^{(1)}_{i, j-\frac{1}{2}}
+\delta_{h,2}^-\,\widetilde{u}^{(2)}_{i-\frac{1}{2}, j}
=\tilde{g}_{i-\frac{1}{2}, j-\frac{1}{2}}
+ \widetilde{R}_{i-\frac{1}{2}, j-\frac{1}{2}}, \; \;{\rm in }\;\; M_h,\label{h3}
\end{align}
\end{subequations}
with the boundary conditions
\begin{equation}
\label{bdh}
\begin{split}
&\widetilde{u}^{(1)}(x, -\frac{h}{2})=-\widetilde{u}^{(1)}(x, \frac{h}{2})
+\mathcal{O}(h^4),\;\;\;\;
\widetilde{u}^{(1)}(x, 1-\frac{h}{2})=-\widetilde{u}^{(1)}(x, 1+\frac{h}{2})
+\mathcal{O}(h^4),\\[4pt]
&\widetilde{u}^{(2)}(-\frac{h}{2}, y)=-\widetilde{u}^{(2)}(\frac{h}{2}, y)
+\mathcal{O}(h^4),\;\;\;\;
\widetilde{u}^{(2)}(1-\frac{h}{2}, y)=-\widetilde{u}^{(2)}(1+\frac{h}{2}, y)
+\mathcal{O}(h^4),\\[4pt]
&\widetilde{u}^{(1)}(0, y)=\widetilde{u}^{(1)}(1, y)=\widetilde{u}^{(2)}(x, 0)
=\widetilde{u}^{(2)}(x, 1),
\end{split}
\end{equation}
where $\widetilde{R}^{(1)}, \widetilde{R}^{(2)}, \widetilde{R}$ hold that 
\begin{equation}
\begin{split}
&|\widetilde{R}^{(1)}|\leq Ch^2(\|u^{(1)}\|_{4,\infty}+\|p\|_{3,\infty}),\;\;\;
|\widetilde{R}^{(2)}|\leq Ch^2(\|u^{(2)}\|_{4, \infty}+\|p\|_{3,\infty}),\\
&|\widetilde{R}|\leq Ch^2(\|u^{(1)}\|_{3, \infty}+\|u^{(2)}\|_{3, \infty}).
\end{split}
\end{equation}
}
\end{lemma}

\begin{figure}[ht!]
\centering
\includegraphics[width=0.36\textwidth]{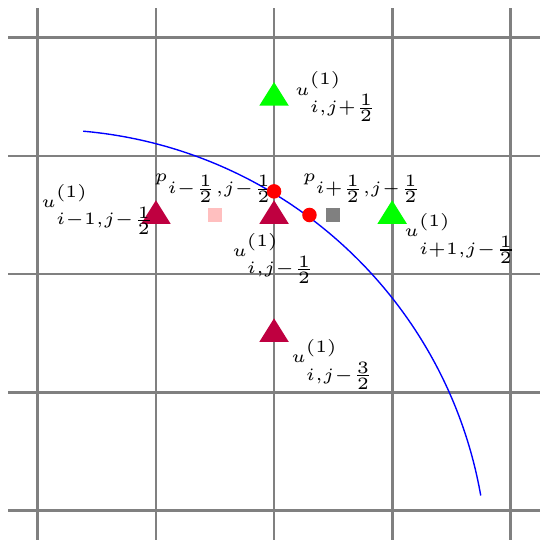}
\setlength{\abovecaptionskip}{-0.0cm}
\setlength{\belowcaptionskip}{-0.4cm} 
\caption{A diagram of the interface cutting through a staggered grid 
around an irregular point}
\label{AccuracyProof}
\end{figure}

\begin{proof}
{\rm The proof is proceeded  in the manner of Strang \cite{MR0166942} and 
Hou \cite{MR1163348, MR1220643}.  Here, only \eqref{h1} is proved in details 
and the proof can be easily extended to \eqref{h2} and \eqref{h3}.  
 
Assume the  grid points $(x_i, y_{j-\frac{1}{2}}), \,(x_{i-1}, y_{j-\frac{1}{2}}),
\, (x_i, y_{j-\frac{3}{2}}),\, (x_{i-\frac{1}{2}},y_{j-\frac{1}{2}})$ locate in the domain 
$\Omega^+$, whereas, $(x_{i+1}, y_{j-\frac{1}{2}}),\, (x_i, y_{j+\frac{1}{2}}),
\, (x_{i+\frac{1}{2}},y_{j-\frac{1}{2}})$ locate in the domain $\Omega^-$.
This is to say, the interface $\Gamma$ intersects the line segment connecting  
grid nodes $(x_i, y_{j-\frac{1}{2}})$ and $(x_{i+\frac{1}{2}}, y_{j-\frac{1}{2}})$ 
at  point $(x^*, y_{j-\frac{1}{2}})$ and intersects the line segment connecting 
grid nodes $(x_i, y_{j-\frac{1}{2}})$ and $(x_i, y_{j+\frac{1}{2}})$ at point 
$(x_i, y^*)$. See Fig. \ref{AccuracyProof} for illustration. 
The left hand of equation \eqref{h1} can be rewritten as
\begin{equation*}
\begin{split}
-\Delta_h\widetilde{u}^{(1)}_{i, j-\frac{1}{2}}+\delta_{h, 1}^+\,
\widetilde{p}_{i-\frac{1}{2}, j-\frac{1}{2}}
=&-\Delta_h\widetilde{u}^{(1)+}_{i, j-\frac{1}{2}}+\delta_{h,1}^+\,
\widetilde{p}^+_{i-\frac{1}{2}, j-\frac{1}{2}}\\[4pt]
&+h^{-2}\left[
\Big(\widetilde{u}^{(1)+}_{i+1, j-\frac{1}{2}}-
\widetilde{u}^{(1)-}_{i+1, j-\frac{1}{2}}\Big)
+\Big(\widetilde{u}^{(1)+}_{i,j+\frac{1}{2}}-
\widetilde{u}^{(1)-}_{i,j+\frac{1}{2}}\Big)
\right]\\[4pt]
&- h^{-1}\Big(\widetilde{p}^+_{i+\frac{1}{2}, j-\frac{1}{2}}
-\widetilde{p}^-_{i+\frac{1}{2}, j-\frac{1}{2}}\Big),
\end{split}
\end{equation*}
where the superscripts $``+"$ and $``-"$ represent the values that are 
one-side limits of the functions from $\Omega^+$ and $\Omega^-$ respectively. 

Expanding the finite differences $\Delta_h\widetilde{u}^{(1)+}_{i, j-\frac{1}{2}}$ and $\delta_{h,1}^+\,\widetilde{p}^+_{i-\frac{1}{2}, j-\frac{1}{2}}$ in Taylor series around the grid node $(x_i, y_{j-\frac{1}{2}})$, one obtains
\begin{equation}
\label{ceq1}
\begin{split}
&-\Delta_{h}\widetilde{u}^{(1)+}_{i, j-\frac{1}{2}}+\delta_{h,1}^+
\,\widetilde{p}^+_{i-\frac{1}{2}, j-\frac{1}{2}}\\[4pt]
&=-h^{-2}
\left(h^2\partial_{xx} u^{(1)+}_{i,j-\frac{1}{2}}
     +h^4\partial_{xx} \hat{u}^{(1)+}_{i,j-\frac{1}{2}}
     +\mathcal{O}(h^4\|u^{(1)}\|_{4, \infty})\right)\\[4pt]
&\quad-h^{-2} 
\left(h^2\partial_{yy} u^{(1)+}_{i,j-\frac{1}{2}}
     +h^4\partial_{yy} \hat{u}^{(1)+}_{i,j-\frac{1}{2}}     
     +\mathcal{O}(h^4\|u^{(1)}\|_{4, \infty})\right)\\[4pt]
&\quad+h^{-1}
\left(h\partial_xp^+_{i-\frac{1}{2}, j-\frac{1}{2}}             
      +h^2\partial_x \hat{p}^+_{i-\frac{1}{2}, j-\frac{1}{2}}      
      +\mathcal{O}(h^3\|p\|_{3, \infty})\right)\\[4pt]
&=\Big(-\partial_{xx} u^{(1)+}_{i,j-\frac{1}{2}}
           -\partial_{yy} u^{(1)+}_{i,j-\frac{1}{2}}           
           +\partial_x p^+_{i-\frac{1}{2}, j-\frac{1}{2}}\Big)
+h^2
\Big(-\partial_{xx} \hat{u}^{(1)+}_{i,j-\frac{1}{2}}       
       -\partial_{yy} \hat{u}^{(1)+}_{i,j-\frac{1}{2}}
       +\partial_x \hat{p}^+_{i-\frac{1}{2}, j-\frac{1}{2}}\Big)\\[4pt]
 &\quad +\mathcal{O}\big(h^2(\|u^{(1)}\|_{4, \infty}+\|p\|_{3, \infty})\big).
\end{split}
\end{equation}
Making Taylor expansion for $\widetilde{u}^{(1)\pm}_{i+1, j-\frac{1}{2}}$ 
around the point $(x^*, y_{j-\frac{1}{2}})$  and 
$\widetilde{u}^{(1)\pm}_{i,j+\frac{1}{2}}$ around the point $(x_i, y^*)$, 
one  arrives  at
\begin{equation}
\label{ceq2}
\begin{split}
&h^{-2}\left[(\widetilde{u}^{(1)+}_{i+1, j-\frac{1}{2}}
                 -\widetilde{u}^{(1)-}_{i+1, j-\frac{1}{2}})
                +(\widetilde{u}^{(1)+}_{i,j+\frac{1}{2}}
                -\widetilde{u}^{(1)-}_{i,j+\frac{1}{2}})\right]\\[4pt]
&=h^{-2}\Big([\![ u^{(1)}]\!]
		  +\xi_{u^{(1)}}[\![ u^{(1)}_x]\!]
                   +\frac{1}{2}\xi_{u^{(1)}}^2[\![ u^{(1)}_{xx}]\!]
                   +\frac{1}{6}\xi_{u^{(1)}}^3[\![ u^{(1)}_{xxx}]\!]
                   +h^2([\![ \hat{u}^{(1)}]\!]
                   +\xi_{u^{(1)}}[\![ \hat{u}^{(1)}_x]\!])   
                   \Big)\\[4pt]
&+h^{-2}
\Big([\![ u^{(1)}]\!]+\eta_{u^{(1)}}[\![ u^{(1)}_y]\!]
      +\frac{1}{2}\eta_{u^{(1)}}^2[\![ u^{(1)}_{yy}]\!]
      +\frac{1}{6}\eta_{u^{(1)}}^3[\![ u^{(1)}_{yyy}]\!]
      +h^2([\![ \hat{u}^{(1)}]\!]
      +\eta_{u^{(1)}}[\![ \hat{u}^{(1)}_y]\!])\Big)\\[4pt]
&+\mathcal{O}(h^2\|u^{(1)}\|_{4, \infty}),
\end{split}
\end{equation}
with $\xi_{u^{(1)}}=x_{i+1}-x^*$ and $\eta_{u^{(1)}}=y_{j+\frac{1}{2}}-y^*$. 
Similarly, making 
Taylor expansions for $\widetilde{p}^{\pm}_{i+\frac{1}{2}, j-\frac{1}{2}}$ 
around the point $(x^*, y_{j-\frac{1}{2}})$, one gets
\begin{equation}
\label{ceq3}
\begin{split}
h^{-1}(\widetilde{p}^+_{i+\frac{1}{2}, j-\frac{1}{2}}
-\widetilde{p}^-_{i+\frac{1}{2}, j-\frac{1}{2}})
=h^{-1}\left([\![ p]\!]+\xi_{p}[\![ p_x]\!]
+\frac{1}{2}\xi_{p}^2[\![ p_{xx}]\!]
+h^2[\![ \hat{p}]\!]+\mathcal{O}(h^3\|p\|_{3, \infty})\right),
\end{split}
\end{equation}
with $\xi_{p} = x_{i+\frac{1}{2}}-x^*$.
Adding \eqref{ceq1}-\eqref{ceq3}  results in
\begin{equation*}
\begin{split}
-\Delta_h\widetilde{u}^{(1)}_{i, j-\frac{1}{2}}+\delta_{h, 1}^+\,
\widetilde{p}_{i-\frac{1}{2}, j-\frac{1}{2}}
&=\tilde{f}_{i,j-\frac{1}{2}}^{(1)}+\mathcal{O}\big(h^2(\|u^{(1)}\|_{4, \infty}+\|p\|_{3, \infty})\big)\\[4pt]
&+h^{2}\Big(
-\partial_{xx} \hat{u}^{(1)+}_{i,j-\frac{1}{2}}       
       -\partial_{yy} \hat{u}^{(1)+}_{i,j-\frac{1}{2}}
       +\partial_x \hat{p}^+_{i-\frac{1}{2}, j-\frac{1}{2}}
\Big)-h^{-1}\Big(
\frac{1}{2}\xi_{p}^2[\![ p_{xx}]\!]
+h^2[\![ \hat{p}]\!]\Big)\\[4pt]
&+h^{-2}\Big(
\frac{1}{6}\xi_{u^{(1)}}^3[\![ u^{(1)}_{xxx}]\!]
                   +h^2\xi_{u^{(1)}}[\![ \hat{u}^{(1)}_x]\!]+h^2[\![ \hat{u}^{(1)}]\!]\Big)\\[4pt]
&+  h^{-2}\Big(                 
                   \frac{1}{6}\eta_{u^{(1)}}^3[\![ u^{(1)}_{yyy}]\!]
      +h^2\eta_{u^{(1)}}[\![ \hat{u}^{(1)}_y]\!]+h^2[\![ \hat{u}^{(1)}]\!]
\Big).
\end{split}
\end{equation*}

To obtain the second order convergence, set the coefficients of powers of $h$ in the last four terms of the above equation as zero. Then one derives that $(\hat{u}^{(1)},\hat{p})$ should satisfy the following equation 
\begin{equation*}
-\Delta \hat{u}^{(1)}+\partial_x\hat{p}=0, \; \hbox{in}\;\Omega,
\end{equation*}
with the jump conditions
\begin{equation*}
\begin{split}
&[\![\hat{u}^{(1)}]\!]=0,\qquad\qquad\qquad
[\![\hat{u}^{(1)}_x]\!]=-\dfrac{1}{6h^2}\xi_{u^{(1)}}^2[\![ u_{xxx}]\!],\\[4pt]
& [\![\hat{p}]\!]=-\dfrac{1}{2h^2}\xi_{p}^2[\![ p_{xx}]\!],\quad\quad
[\![\hat{u}^{(1)}_y]\!]=-\dfrac{1}{6h^2}\eta_{u^{(1)}}^2[\![ u_{yyy}]\!],
\end{split} \qquad\;\hbox{on}\;\Gamma.
\end{equation*}

Now consider the boundary conditions.
Expanding $\widetilde{u}^{(1)}_{i, \pm\frac{1}{2}}$ at grid point $(x_i, 0)$, 
one has
\begin{align*}
&\widetilde{u}^{(1)}_{i,-\frac{1}{2}}=u^{(1)}_{i,0}-\frac{h}{2}\partial_x u^{(1)}_{i,0}
+\frac{h^2}{8}\partial_{xx} u^{(1)}_{i,0}-\frac{h^3}{48}\partial_{xxx} u^{(1)}_{i,0}
+h^2\left(\hat{u}^{(1)}_{i,0}-\frac{h}{2}\partial_x \hat{u}^{(1)}_{i,0}\right)
+\mathcal{O}(h^4\|u^{(1)}\|_{4, \infty}),\\[4pt]
&\widetilde{u}^{(1)}_{i,\frac{1}{2}}=u^{(1)}_{i,0}+\frac{h}{2}\partial_x u^{(1)}_{i,0}
+\frac{h^2}{8}\partial_{xx} u^{(1)}_{i,0}+\frac{h^3}{48}\partial_{xxx} u^{(1)}_{i,0}
+h^2\left(\hat{u}^{(1)}_{i,0}+\frac{h}{2}\partial_x \hat{u}^{(1)}_{i,0}\right)
+\mathcal{O}(h^4\|u^{(1)}\|_{4, \infty}).
\end{align*} 
Thus 
\begin{equation}
\widetilde{u}^{(1)}_{i,-\frac{1}{2}}+\widetilde{u}^{(1)}_{i, \frac{1}{2}}
=2u^{(1)}_{i,0}+\frac{h^2}{4}\partial_{xx} u^{(1)}_{i,0}
+2h^2\hat{u}^{(1)}_{i,0}+\mathcal{O}(h^4\|u^{(1)}\|_{4, \infty}).
\end{equation}
Recalling that $u^{(1)}_{i,0}=0$, one gets
\begin{equation}
\tilde{u}^{(1)}_{i,-\frac{1}{2}}=-\tilde{u}^{(1)}_{i, \frac{1}{2}}
+\mathcal{O}(h^4\|u^{(1)}\|_{4, \infty}),
\end{equation}
with $\hat{u}^{(1)}_{i,0}=-\dfrac{1}{8}\partial_{xx}u^{(1)}_{i,0}$. 

The proof for other equations in \eqref{bdh} can be obtained similarly, which is omitted here.

The above proof mainly focuses on the irregular grid nodes 
$(x_i, y_{j-\frac{1}{2}})\in \Omega^+$, and the results at other regular and irregular grid
nodes can be derived similarly.  
The proof is completed. }
\end{proof}

\begin{remark}
Based on the above proof, one can find the functions 
$\hat{u}^{(1)}, \hat{u}^{(2)}, \hat{p}$ satisfy
\begin{equation*}
\begin{split}
-\Delta  \hat{\vect   u}+\nabla \hat{p}=0&,\; \; {\rm on}\;\Omega,\\[4pt]
\nabla\cdot\hat{\vect   u}=0&,\; \; {\rm on}\;\Omega,\\[4pt]
\hat{\vect   u}=-\frac{1}{8}\Delta \vect   u&, \;\; {\rm on}\;\partial\Omega,
\end{split}
\end{equation*}
with jump conditions
\begin{equation}
\label{hr}
\begin{split}
&[\![ \hat{\vect   u}]\!] =\vect 0, \qquad \;\;
[\![ \hat{p}]\!] = l_0(x)[\![ p_{xx}]\!]
+l_0(y)[\![ p_{yy}]\!],\\[4pt]
&[\![ \hat{u}^{(1)}_x]\!] 
= l_1(x)[\![ u^{(1)}_{xxx}]\!],\quad\,
[\![ \hat{u}^{(1)}_y]\!] 
= l_1(y)[\![ u^{(1)}_{yyy}]\!],     \qquad \hbox{on}\; \Gamma,\\[4pt]
&[\![ \hat{u}^{(2)}_x]\!] 
= l_2(x)[\![ u^{(2)}_{xxx}]\!],\quad\;
[\![ \hat{u}^{(2)}_y]\!] = l_2(y)[\![ u^{(2)}_{yyy}]\!], 
\end{split}
\end{equation}
where $|l_i|\leq 1, i=0,1,2$ and they only involve the location of interface 
$\Gamma$.
Note that the jump conditions of  high-order 
derivatives on the right-hand side of \eqref{hr} exist and can 
be obtained using the similar way in section \ref{sec; jump}.
\end{remark}

For brevity, define the error functions 
\begin{equation*}
\begin{split}
&\widetilde{e}_u^{(1)}=u^{(1)}_h-\widetilde{u}^{(1)} \in V_h^{(1)},\quad
\widetilde{e}_u^{(2)}=u^{(2)}_h-\widetilde{u}^{(2)} \in V_h^{(2)},\quad
\widetilde{e}_p=p_h-\widetilde{p}\in M_h.
\end{split}
\end{equation*}
It is easy to verify that $(\widetilde{e}_u^{(1)},
\widetilde{e}_u^{(2)}, \widetilde{e}_p)$ satisfy the following error equations 
\begin{equation}
\label{error}
\begin{split}
-\Delta_h\widetilde{e}_u^{(1)}+\delta_{h,1}^+\,\widetilde{e}_p
&= \widetilde{R}_u^{(1)}, \;\;\;\;\hbox{in} \;V_h^{(1)}, \\
-\Delta_{h}\widetilde{e}_u^{(2)}+\delta_{h, 2}^+\,\widetilde{e}_p
&=\widetilde{R}_u^{(2)},\; \;\;\;\hbox{in} \;V_h^{(2)},\\
\delta_{h,1}^-\,\widetilde{e}_u^{(1)}+\delta_{h,2}^-\,\widetilde{e}_u^{(2)}
&=\widetilde{R},\;\;\; \;\,\;\;\;\hbox{in} \;M_h.
\end{split}
\end{equation}

It is  pointed out that the truncation errors in \eqref{error}  are of second-order 
accuracy at all the computational points. The optimal second-order 
convergence in $\ell^2$-norms comes straightforwardly.

\begin{theorem}
\label{Ihmaintheo}
{\em Suppose that the analytical solutions $(u^{(1)}, u^{(2)}, p)$ are sufficiently 
smooth on $\Omega$ excluding $\Gamma$, 
$(\widetilde{u}^{(1)}, \widetilde{u}^{(2)}, \widetilde{p})$ are defined 
by \eqref{diuvp} in Lemma \ref{consistence}.
There exists a positive constant $C$ independent of $h$ such that it holds the following 
discrete $H^1$-error estimate
\begin{equation}
\label{nablaL2}
\begin{split}
&|\widetilde{\vect   e}_u|_1
\leq Ch^2(\|\vect   u\|_{4, \infty}+\|p\|_{3, \infty}),
\end{split}
\end{equation}
and the following discrete $\ell^2$-error estimates
\begin{subequations}
\begin{align}
\|\widetilde{\vect   e}_u\|\leq Ch^2(\|\vect   u\|_{4, \infty}+\|p\|_{3, \infty}),\label{uvL2}\\[4pt]
\|\widetilde{e}_p\|_{M_h}\leq Ch^2(\|\vect   u\|_{4, \infty}+\|p\|_{3, \infty}),\label{pL2}
\end{align}
\end{subequations}
with $\widetilde{\vect   e}_u=(\widetilde{e}_u^{(1)}, \widetilde{e}_u^{(2)})$.}
\end{theorem}
\begin{proof} {\em 
Computing the discrete inner-product of 
\eqref{h1} and \eqref{h2}  with the discrete function $h^2e_v^{(1)}\in V_h^{(1)}$ 
and $h^2e_v^{(2)}\in V_h^{(2)}$, then adding the results and using the discrete 
Green formulae in Lemma \ref{dGf}, one obtains
\begin{align}
\big(\delta_{h,1}^-\,\widetilde{e}_u^{(1)}, 
&\delta_{h,1}^-\,e_v^{(1)}\big)_{M_h}
+\big(\delta_{h,2}^-\,\widetilde{e}_u^{(1)}, 
\delta_{h,2}^-\,e_v^{(1)}\big)_{W_h^{(1)}}
+\big(\delta_{h,1}^-\,\widetilde{e}_u^{(2)},
\delta_{h, 1}^-\,e_v^{(2)}\big)_{W_h^{(2)}}
+\big(\delta_{h,2}^-\,\widetilde{e}_u^{(2)}, 
\delta_{h,2}^-\,e_v^{(2)}\big)_{M_h} \nonumber\\
&\quad-\big(\widetilde{e}_p, \delta_{h,1}^-\,e_v^{(1)}
+\delta_{h, 2}^-\,e_v^{(2)}\big)_{M_h}
=-\big[\big(\widetilde{R}_u, e_v^{(1)}\big)_{V_h^{(1)}}
+\big(\widetilde{R}_v, e_v^{(2)}\big)_{V_h^{(2)}}\big].  \label{Ieq}
\end{align}
By applying Cauchy-Schwarz inequality and discrete Poincare inequality 
in Lemma \ref{disPoin}, one derives
\begin{align}
\big(\widetilde{R}^{(1)}_u, e_v^{(1)}\big)_{V_h^{(1)}}
&\leq C_1\|\widetilde{R}_u^{(1)}\|_{V_h^{(1)}}\|e_v^{(1)}\|_{V_h^{(1)}}\nonumber\\
&\leq C_2\|\widetilde{R}_u^{(1)}\|_{V_h^{(1)}}
\Big(\|\delta_{h,1}^-\,e_v^{(1)}\|_{M_h}
+\|\delta_{h,2}^-\,e_v^{(1)}\|_{W_h^{(1)}}\Big), \label{leq1} \\
\big(\widetilde{R}_u^{(2)}, e_v^{(2)}\big)_{V_h^{(2)}}
&\leq C_3\|\widetilde{R}_u^{(2)}\|_{V_h^{(2)}}
\|e_v^{(2)}\|_{V_h^{(2)}} \nonumber\\
&\leq C_4\|\widetilde{R}_u^{(2)}\|_{V_h^{(2)}}
\Big(\|\delta_{h,1}^-\,e_v^{(2)}\|_{W_h^{(2)}}
+\|\delta_{h, 2}^-\,e_v^{(2)}\|_{M_h}\Big). \label{leq2}
\end{align}
Thus,  combining \eqref{Ieq}-\eqref{leq2}, one obtains
\begin{equation}
\label{leq4}
\begin{split}
&\big(\widetilde{e}_p, \delta_{h,1}^-\,e_v^{(1)}
+\delta_{h,2}^-\,e_v^{(2)}\big)_{M_h}\leq
C_5\left(|\widetilde{\vect   e}_u|_1
+\|\widetilde{R}^{(1)}_u\|_{V_h^{(1)}}
+\|\widetilde{R}_u^{(2)}\|_{V_h^{(2)}}\right)|\vect   e_v|_1.
\end{split}
\end{equation}
Using the discrete LBB condition in Lemma \ref{LBB} 
and inequality \eqref{leq4}, one gets
\begin{equation}
\label{Ihep}
\begin{split}
\|\widetilde{e}_p\|_{M_h} 
&\leq\sup\limits_{\vect   e_v\in\vect   V_h}\frac{\big|(\widetilde{e}_p, 
\delta_{h,1}^-\,e_v^{(1)}+\delta_{h,2}^-\,e_v^{(2)})_{M_h}\big|}
{|\vect   e_v|_1}\\[4pt]
&\leq C_{6}\left(|\widetilde{\vect   e}_u|_1
+\|\widetilde{R}_u^{(1)}\|_{V_h^{(1)}}
+\|\widetilde{R}_u^{(2)}\|_{V_h^{(2)}}\right)\\[4pt]
& \leq C_{7}\left(|\widetilde{\vect   e}_u|_1
+h^2(\| u^{(1)}\|_{4, \infty}+\|p\|_{3, \infty})
+h^2(\| u^{(2)}\|_{4, \infty}+\|p\|_{3, \infty})\right).
\end{split}
\end{equation}
 Setting $e_v^{(1)}=\widetilde{e}_u^{(1)}, e_v^{(2)}=\widetilde{e}_u^{(2)}$ 
 in \eqref{Ieq}, one has
 \begin{equation}
 \label{leqe}
 |\tilde{\vect   e}|_1^2
 =\big(\widetilde{e}_p, \delta_{h,1}^-\,\widetilde{e}_u^{(1)}
 +\delta_{h,2}^-\,\widetilde{e}_u^{(2)}\big)_{M_h}
 +\big(\widetilde{R}_u^{(1)}, \widetilde{e}_u^{(1)}\big)_{V_h^{(1)}}
 +\big(\widetilde{R}_u^{(2)}, \widetilde{e}_u^{(2)}\big)_{V_h^{(2)}}.
 \end{equation}
 Using the same technique, one derives
 \begin{equation}
  \label{leqe2}
 \begin{split}
\big(\widetilde{R}_u^{(1)}, \widetilde{e}_u^{(1)}\big)_{V_h^{(1)}}
+\big(\widetilde{R}_u^{(2)}, \widetilde{e}_u^{(2)}\big)_{V_h^{(2)}}
&\leq \frac{1}{4}|\widetilde{\vect   e}_u|_1^2
+C_{8}\left(\|\widetilde{R}_u^{(1)}\|_{V_h^{(1)}}
+\|\widetilde{R}_u^{(2)}\|_{V_h^{(2)}}\right)\\[4pt]
& \leq \frac{1}{4}| \widetilde{\vect   e}_u|_1^2
+C_{9}h^4\big(\|\vect   u\|_{4, \infty}+\|p\|_{3, \infty}\big)^2.
\end{split}
 \end{equation}
 Moreover, from \eqref{h3}, \eqref{leq4} and the discrete Green formulae 
 in Lemma \ref{dGf}, one gets
 \begin{equation}
  \label{leqe3}
 \begin{split}
\big(\widetilde{e}_p, \delta_{h,1}^-\,\tilde{e}_u^{(1)}
+\delta_{h,2}^-\,\widetilde{e}_u^{(2)}\big)_{M_h}
 &\leq C\|\widetilde{e}_p\|_{M_h}\|\widetilde{R}\|_{M_h}\\[4pt]
 &\leq \frac{1}{4}|\widetilde{\vect   e}_u|_1^2
 +C_{10}h^4\big(\|\vect   u\|_{4, \infty}+\|p\|_{3, \infty}\big)^2.
 \end{split}
 \end{equation}
 Therefore, combining \eqref{leqe},\eqref{leqe2} and \eqref{leqe3}, one 
 obtains
 \begin{equation}
 \label{Ihe}
 |\widetilde{\vect   e}_u|_1\leq  Ch^2(\|\vect   u\|_{4, \infty}+\|p\|_{3, \infty}).
 \end{equation}
 Then,  \eqref{uvL2} follows from the discrete Poincare inequality in 
 Lemma \ref{disPoin}. Furthermore,  \eqref{pL2}
 comes straightforwardly from \eqref{Ihep} and \eqref{Ihe}.}
\end{proof}

Denote the error functions 
\begin{equation*}
\begin{split}
&e_u^{(1)}=u^{(1)}_h-u^{(1)}\in V_h^{(1)},\quad
e_u^{(2)}=u^{(2)}_h-u^{(2)}\in V_h^{(2)},\quad
e_p=p_h-p\in M_h.
\end{split}
\end{equation*}
In terms of the definition of $\widetilde{\vect   e}_u, \widetilde{e}_p$, 
it is obvious that
\begin{equation*}
\begin{split}
\vect   e_u &=\widetilde{\vect   e}_u +  h^2 \hat{\vect   u},\qquad
e_p = \widetilde{e}_p + h^2 \hat{p}.
\end{split}
\end{equation*}
with $\vect   e_u=\big(e_u^{(1)}, e_u^{(2)}\big)$.
 The following $\ell^2$-analysis comes straightforwardly. 
\begin{theorem}
\label{maintheo}
{\em Suppose that the analytical solutions $(u^{(1)}, u^{(2)}, p)$ are sufficiently 
smooth on $\Omega$ excluding $\Gamma$,  $(u_h^{(1)}, u_h^{(2)}, p_h)$ 
are numerical solutions defined in \eqref{MAC}.  There exists a positive 
constant C independent of $h$ such that
\begin{equation*}
\begin{split}
|\vect   e_u|_1&\leq Ch^2(\|\vect   u\|_{4, \infty}+\|p\|_{3, \infty}),\\[4pt]
\|\vect   e_u\|&\leq Ch^2(\|\vect   u\|_{4, \infty}+\|p\|_{3, \infty}),\\[4pt]
\|e_p\|_{M_h}&\leq Ch^2(\|\vect   u\|_{4, \infty}+\|p\|_{3, \infty}).
\end{split}
\end{equation*}
}
\end{theorem}  

It is remarked that the analysis is given for the two dimensional problem, 
but similar results can be obtained easily for three dimensional problems. 
Because of the size limitation, the detailed derivation is not illustrated here, 
but the numerical examples for three dimensional case is presented later, 
which indicates that the results are consistent with that in two dimension.

\section{Numerical Examples}
In this section, numerical results are presented to verify the theoretical 
analysis. To evaluate convergence rates, define the scaled discrete 
$\ell^2$-norms:
$$\|e_{\vect   u}\|=\frac{\|\vect   u_h-\vect   u\|}{\|\vect   u\|},\quad\;
\|e_{p}\|_{M_h}=\frac{\|p_h-p\|_{M_h}}{\|p\|_{M_h}},\quad\;
|e_{\vect   u}|_{1}=\frac{|\vect   u_h-\vect   u|_1}{|\vect   u|_1},$$
and  the scaled discrete maximum norms:  
$$\|e_{\vect   u}\|_{\infty}=\frac{\|\vect   u_h-\vect   u\|_{\infty}}{\|\vect   u\|_{\infty}},\quad\;
|e_{\vect   u}|_{1,\infty}=\frac{|\vect   u_h-\vect   u|_{1, \infty}}{|\vect   u|_{1, \infty}},$$
where $|\vect   v|_{1, \infty}=\max|\delta_{h, i}^{\pm}v^{(j)}|$ with $i,j =1, 2$.

{\em Example 1.} In this example,  the interface is a circle with $r=1$, 
which is located at the center of the box $\Omega=(-2, 2)^2$. 
The solution is given by 
\begin{equation*}
\begin{split}
u^{(1)}(x, y)&=\begin{cases} \dfrac{y}{r}-y+\dfrac{y}{4}, \; x^2+y^2>1,\\[6pt]
\dfrac{y}{4}(x^2+y^2),\;x^2+y^2\leq 1,
\end{cases}\\[4pt]
u^{(2)}(x, y)&=\begin{cases} -\dfrac{x}{r}+x-\dfrac{x}{4}(1-x^2), 
\; x^2+y^2>1,\\[4pt]
-\dfrac{xy^2}{4},\qquad\qquad\quad\;\;\;\;\; \,x^2+y^2\leq 1,
\end{cases}\\[4pt]
p(x, y)&=\begin{cases} (-\dfrac{3}{4}x^3+\dfrac{3}{8}x)y, 
\; x^2+y^2>1,\\[4pt]
5,\qquad\qquad\quad\;\;\;\, x^2+y^2\leq 1.
\end{cases}
\end{split}
\end{equation*}
\begin{table}[h]
\caption{$\ell^2$-errors and  its  convergence rates of {\em Example} 1. }
\begin{center}
\vspace{-0.1cm}
\begin{tabular}{|c|c|c|c|c|c|c|}
\hline
grid size &  $\|e_{\vect   u}\|$
& order & $\|e_p\|$ & order &$|e_{\vect   u}|_{1}$ &order\\
\hline
$128 \times  128$   &3.77e-3  & -        & 5.50e-5  & -        & 1.43e-4 &      -  \\
$256 \times  256$   &9.57e-4  &  1.99 & 1.39e-5  & 1.98  & 3.60e-5 & 1.99 \\
$512 \times  512$   &2.36e-4  &  2.02 & 3.36e-6  & 2.05  & 9.00e-6 & 2.00 \\
$1024\times 1024$ &5.91e-5  &  2.00 & 8.44e-7  & 1.99	& 2.25e-6 & 2.00\\
$2048\times 2048$ &1.48e-5  &  2.00 & 2.11e-7  & 2.00	& 5.63e-7 & 2.00 \\	
\hline
\end{tabular}
\end{center}
\label{tabcircle}
\end{table}

 \begin{table}[h]
\caption{Maximum errors and its convergence rates of {\em Example} 1.}
\begin{center}
\vspace{-0.1cm}
\begin{tabular}{|c|c|c|c|c|c|c|}
\hline
grid size  &$\|e_{\vect   u}\|_{\infty}$ 
 & order & $|e_{\vect   u}|_{1,\infty}$& order\\
\hline
$128 \times  128$   &1.34e-4  & -         & 1.42e-4 &      -  \\
$256 \times  256$   &3.38e-5  &  2.11  & 3.51e-5 & 2.11 \\
$512 \times  512$   &8.50e-6  &  1.99  & 8.73e-5 & 2.01 \\
$1024\times 1024$ &2.13e-6  &  2.00  & 2.17e-6 & 2.01\\
$2048\times 2048$ &5.34e-7  &  2.00  & 5.44e-7 & 1.99 \\	
\hline
\end{tabular}
\end{center}
\label{tabcirclemax}
\end{table}
The errors and convergence rates in the $\ell^2$-norms and  maximum norms 
are shown in Table \ref{tabcircle} and  Table \ref{tabcirclemax}, respectively, 
which indicate that the velocity and its gradient are of second-order accuracy 
in both the discrete  $\ell^2$-norm and the discrete maximum norm, 
 and the pressure is also second order accurate in the $\ell^2$-norm.  Fig. \ref{Ex:1} shows the 
 solution plots of the $x$-component of the velocity $u^{(1)}$, 
 the $y$-component of the velocity $u^{(2)}$ and the pressure $p$.  These numerical results verify the theoretical analysis.  

\begin{figure}[ht!]
\centering
\subfigure[the velocity field $u^{(1)}$]{
\includegraphics[width=0.31\textwidth]{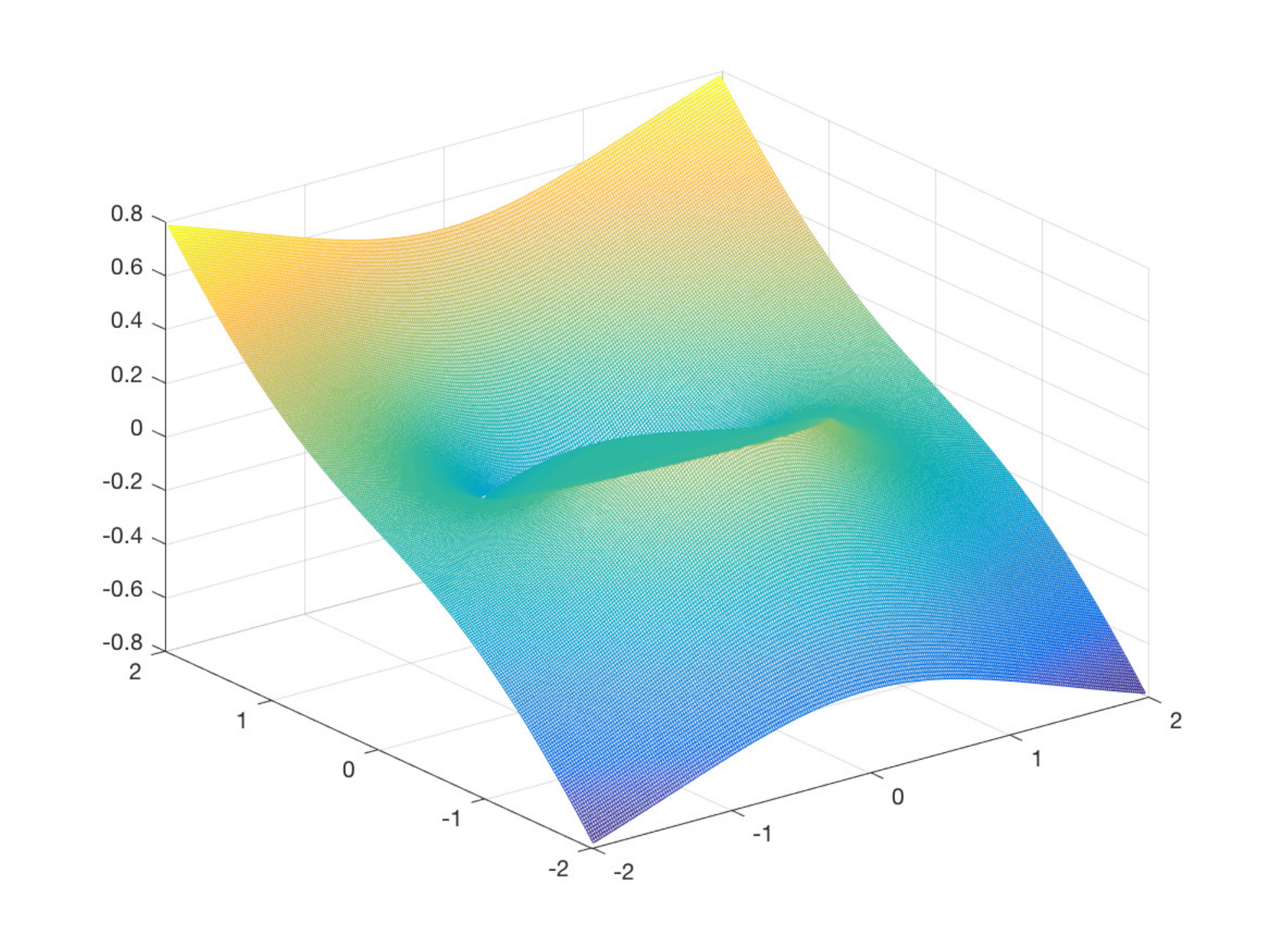}
}
\subfigure[the velocity field $u^{(2)}$]{
\includegraphics[width=0.31\textwidth]{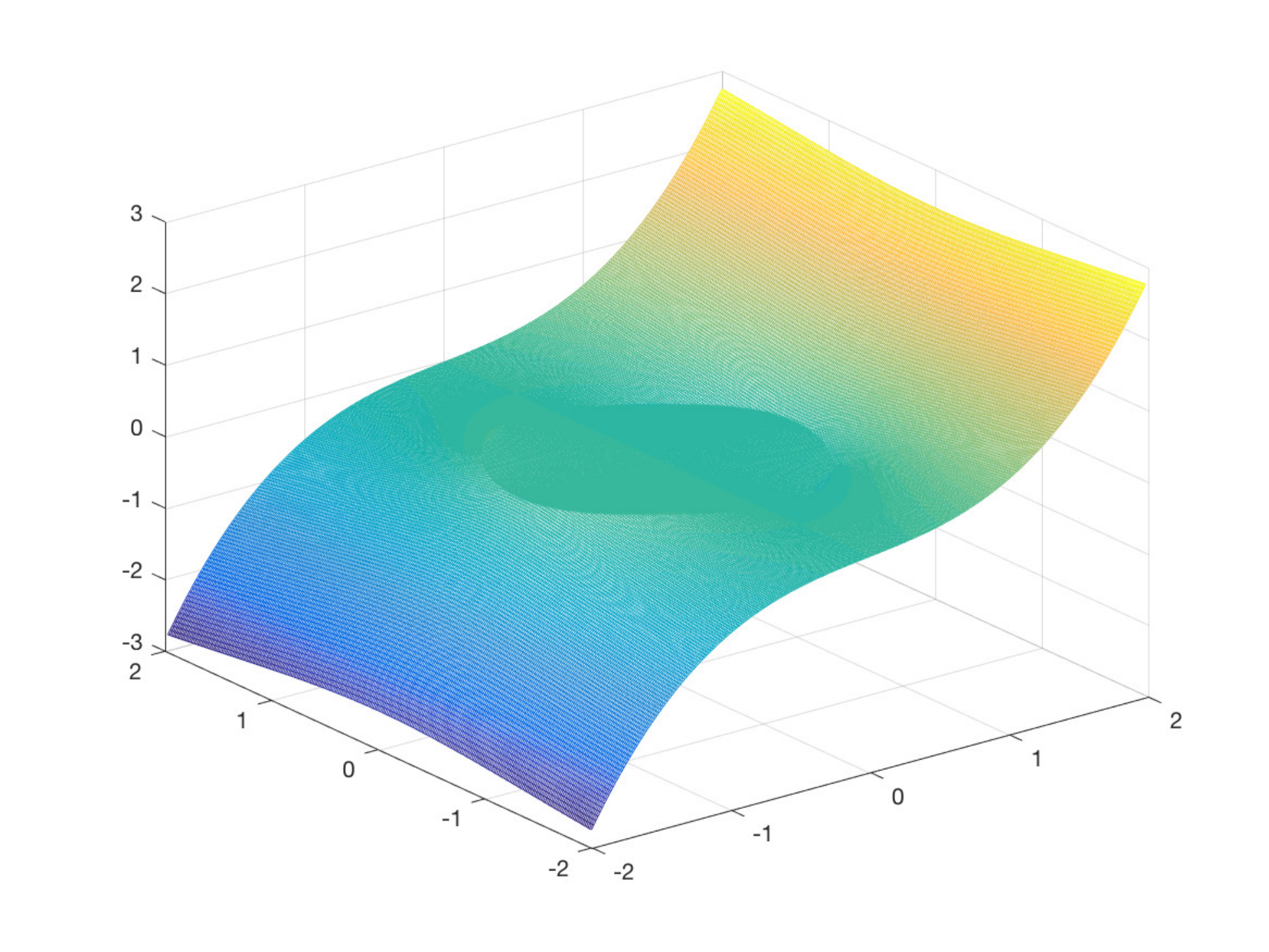}
}
\subfigure[the pressure field $p$]{
\includegraphics[width=0.31\textwidth]{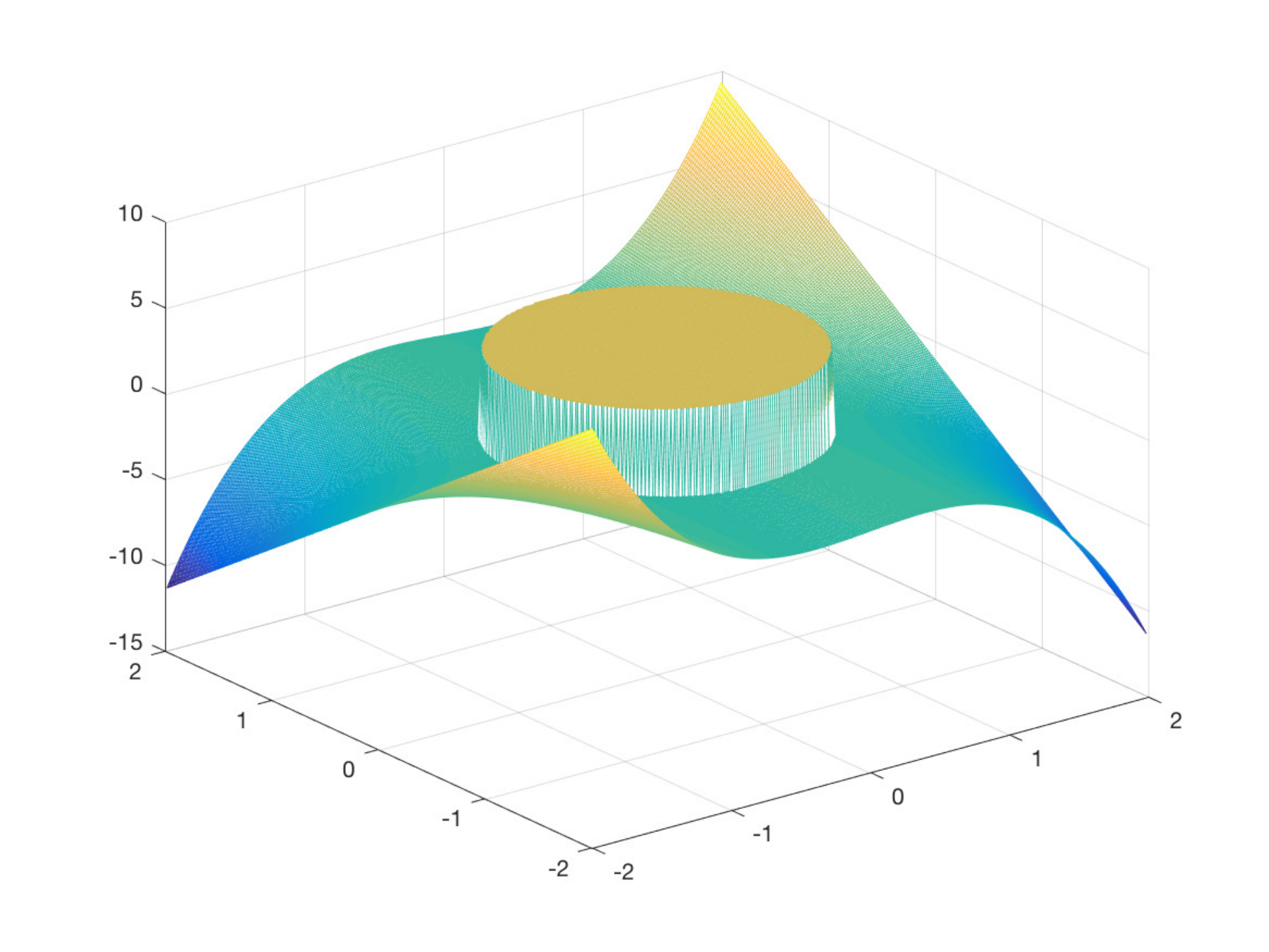}
}
\setlength{\abovecaptionskip}{-0.0cm}
\setlength{\belowcaptionskip}{-0.0cm}
 \caption{Numerical solutions for example 1 on a $128\times 128$ 
 grid.} 
 \label{Ex:1}
\end{figure}

{\em Example 2.} In this example, the interface is an ellipse which is  
governed by $x^2+4y^2 =1$  and the computational domain is
$\Omega=(-2, 2)^2$. The exact velocity and pressure are given by 
\begin{equation*}
\begin{split}
u^{(1)}(x, y)&=\begin{cases}
\dfrac{y}{4}(x^2+4y^2),\;x^2+4y^2\geq 1,\\[6pt]
\dfrac{y}{4}, \qquad\qquad\;\;\; x^2+4y^2<1,
\end{cases}\\
u^{(2)}(x, y)&=\begin{cases}
-\dfrac{xy^2}{4},\qquad\quad\,x^2+4y^2\geq 1,\\[4pt]
-\dfrac{x}{16}(1-x^2), \; x^2+4y^2<1,
\end{cases}\\
p(x, y)&=\begin{cases} 
0,\qquad\qquad\quad\;\;\;\, x^2+4y^2\geq 1,\\[4pt]
(-\dfrac{3}{4}x^3+\dfrac{3}{8}x)y, \; x^2+4y^2<1.
\end{cases}
\end{split}
\end{equation*}
\begin{table}[h]
\caption{$\ell^2$-errors and  its  convergence rates of {\em Example} 2. }
\begin{center}
\vspace{-0.1cm}
\begin{tabular}{|c|c|c|c|c|c|c|c|}
\hline
grid size & $\|e_{\vect   u}\|$
& order & $\|e_p\|$ & order &$\|e_{\vect   u}\|_1$ &order\\
\hline
$128 \times  128$   &3.53e-3  & -        & 2.41e-5  & -        & 2.49e-4 &      -  \\
$256 \times  256$   &8.87e-4  &  1.99 & 5.11e-6  & 2.24  & 6.14e-5 & 2.02 \\
$512 \times  512$   &2.20e-4  &  2.01 & 1.08e-6  & 2.24  & 1.53e-5 & 2.00\\
$1024\times 1024$ &5.53e-5  &  1.99 & 1.95e-7  & 2.47	& 3.82e-6 & 2.00\\
$2048\times 2048$ &1.39e-5  &  1.99 & 3.33e-8  & 2.55	& 9.55e-7 & 2.00 \\		
\hline
\end{tabular}
\end{center}
\label{tabellipse}
\end{table}

 \begin{table}[h]
\caption{Maximum errors and its convergence rates of {\em Example} 2.}
\begin{center}
\vspace{-0.1cm}
\begin{tabular}{|c|c|c|c|c|c|c|c|}
\hline
grid size &$\|e_{\vect   u}\|_{\infty}$ 
 & order & $|e_{\vect   u}|_{1,\infty}$& order\\
\hline
$128 \times  128$   & 2.44e-4  & -         & 5.18e-5 &      -  \\
$256 \times  256$   &6.09e-5  &  2.00  & 1.27e-5 & 2.02 \\
$512 \times  512$   &1.52e-5  &  2.00  & 3.17e-6 & 2.00 \\
$1024\times 1024$ &3.81e-6  &  2.00  & 7.92e-7 & 2.00\\
$2048\times 2048$ &9.53e-7  &  2.00  & 1.98e-7 & 2.00 \\	
\hline
\end{tabular}
\end{center}
\label{tabellipsemax}
\end{table}

The plots of the solutions are shown in Fig. \ref{Ex:2}.  The second order 
accurate solutions for the velocity, the pressure as well as  the gradient of
 the velocity in  discrete $\ell^2$-norms are displayed  in Table \ref{tabellipse}, 
 and the second order accurate solutions 
 in the maximum norms are displayed in Table \ref{tabellipsemax}.
\begin{figure}[ht!]
\centering
\subfigure[the velocity field $u^{(1)}$]{
\includegraphics[width=0.31\textwidth]{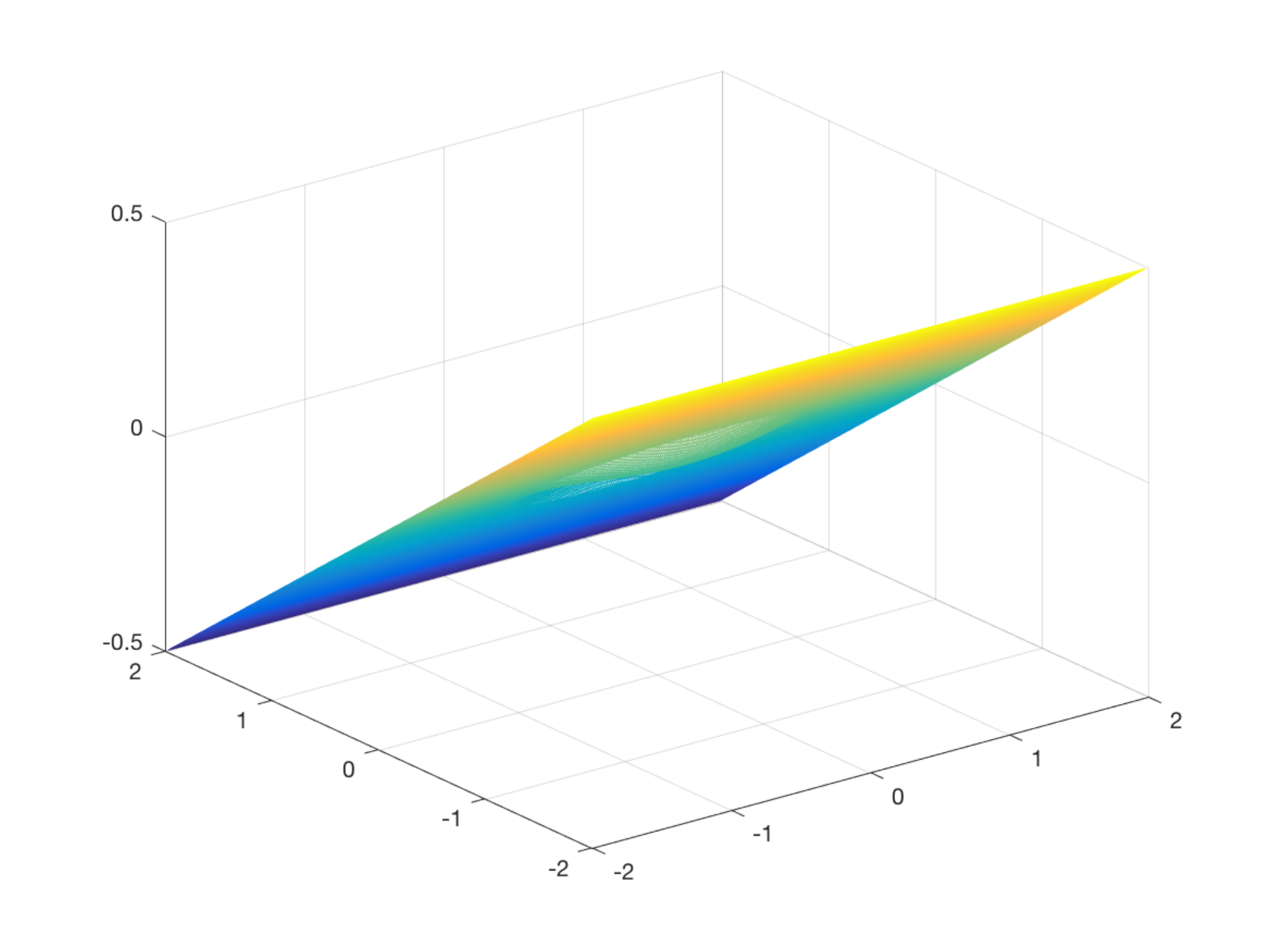}
}
\subfigure[the velocity field $u^{(2)}$]{
\includegraphics[width=0.31\textwidth]{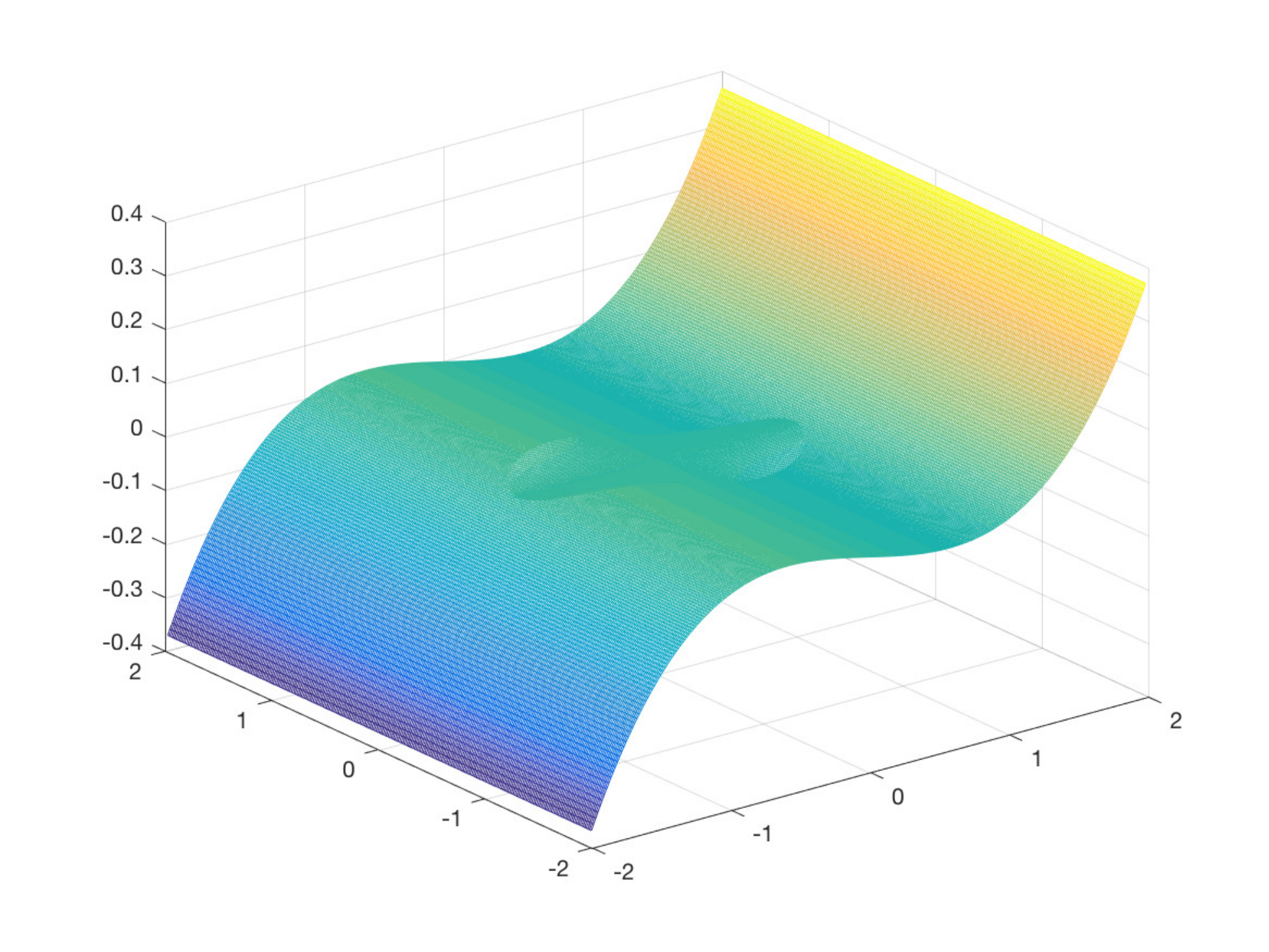}
}
\subfigure[the pressure field $p$]{
\includegraphics[width=0.31\textwidth]{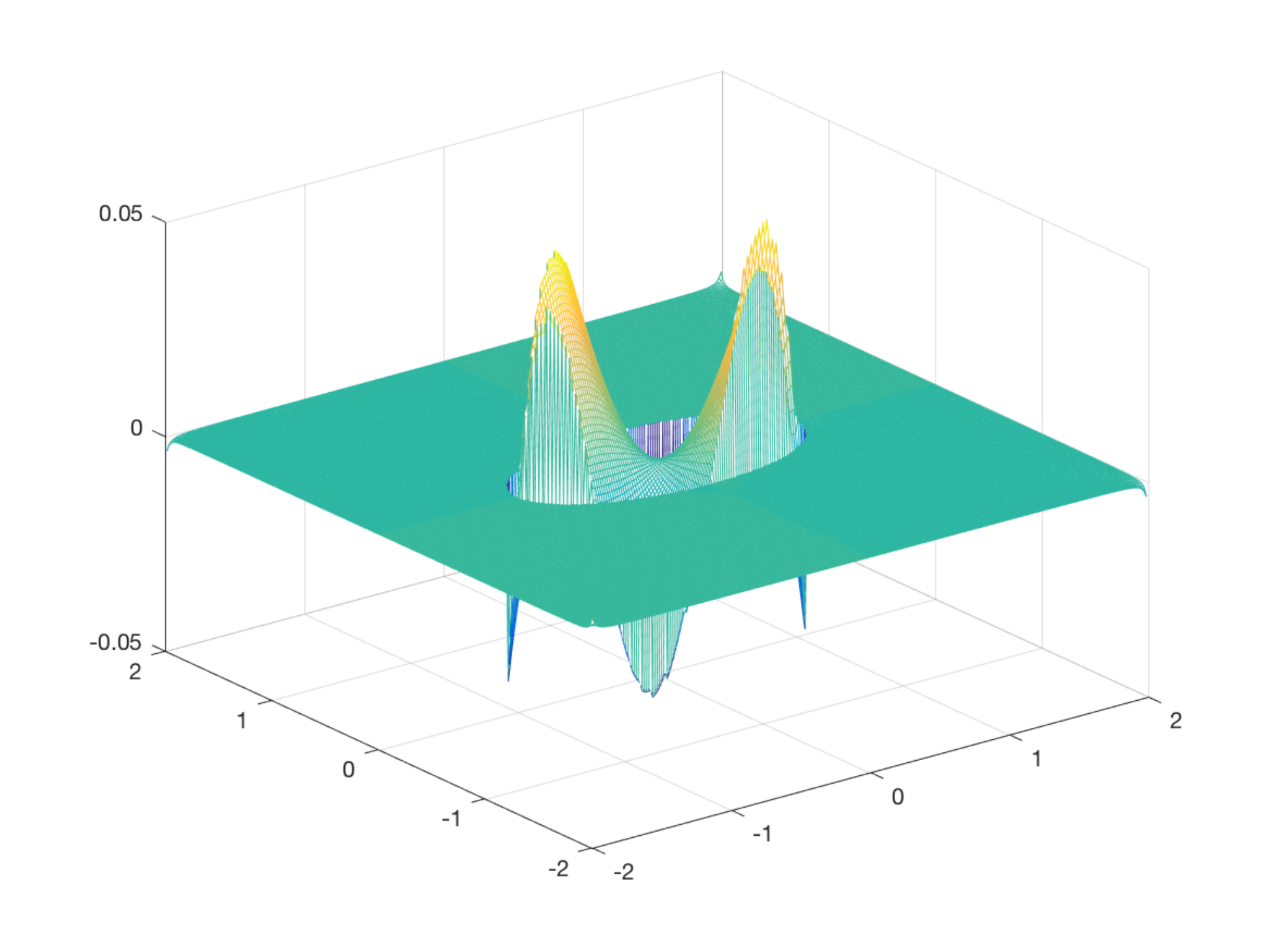}
}
 \setlength{\abovecaptionskip}{-0.0cm}
\setlength{\belowcaptionskip}{-0.0cm}
\caption{Numerical solution for example 2 on a $128\times 128$ 
 grid.} 
 \label{Ex:2}
\end{figure}

{\em Example 3.}  In order to illustrate the second accuracy of the Stokes solver for more complicated case, a three-dimensional problem is presented in the last example. The velocity and pressure are all discontinuous across the interface, 
which is a sphere with $r=1$ and located at the center of 
the box $\Omega=(-2, 2)^3$. 
The solution is given by 
\begin{equation*}
\begin{split}
u^{(1)}(x, y,z)&=\begin{cases} \exp(\cos y) + \exp(\sin z), &\; x^2+y^2+z^2>1,\\[6pt]
-4(1-x^2-y^2)xy-4x^2z^2+(x^2+3z^2-2)(z^2-x^2),&\;x^2+y^2+z^2\leq 1,
\end{cases}\\[4pt]
u^{(2)}(x, y,z)&=\begin{cases} \exp(\sin x), &\; x^2+y^2+z^2>1,\\[6pt]
-4x^2y^2+(3x^2+y^2-2)(x^2-y^2),&\;x^2+y^2+z^2\leq 1,
\end{cases}\\[4pt]
u^{(3)}(x, y,z)&=\begin{cases} \exp(\cos(x)), 
&\; x^2+y^2+z^2>1,\\[6pt]
-4(1-x^2-z^2)xz,&\;x^2+y^2+z^2\leq 1,
\end{cases}\\[4pt]
p(x, y,z)&=\begin{cases} \exp(\cos x+\sin y )+\exp(\cos z +\sin x),  &x^2+y^2+z^2>1,\\[4pt]
(x-1)^3+(y-1)^3+(z-1)^2,& x^2+y^2+z^2\leq 1.
\end{cases}
\end{split}
\end{equation*}
\begin{table}[h]
\caption{$\ell^2$-errors and  its  convergence rates of {\em Example} 3. }
\begin{center}
\vspace{-0.1cm}
\begin{tabular}{|c|c|c|c|c|c|c|}
\hline
grid size &  $\|e_{\vect   u}\|$
& order & $\|e_p\|$ & order &$|e_{\vect   u}|_{1}$ &order\\
\hline
$128 \times  128\times  128$   &1.79e-4  & -        & 1.82e-3 &  -        & 1.31e-4   &    -  \\
$256 \times  256\times  256$   &4.53e-5  & 1.98 & 5.32e-4  & 1.77  & 2.57e-5  &  2.35\\
$512 \times  512\times  512$   &1.12e-5  &  2.02 & 1.48e-4  & 1.85 & 4.97e-6 & 2.37 \\
\hline
\end{tabular}
\end{center}
\label{tabsphere}
\end{table}

 \begin{table}[h]
\caption{Maximum errors and its convergence rates of {\em Example} 3.}
\begin{center}
\vspace{-0.1cm}
\begin{tabular}{|c|c|c|c|c|c|c|}
\hline
grid size  &$\|e_{\vect   u}\|_{\infty}$ 
 & order & $|e_{\vect   u}|_{1,\infty}$& order\\
\hline
$128 \times  128\times  128$   &2.55e-4  &  -  & 1.97e-4 & -\\
$256 \times  256\times  256$   &6.30e-5  &  2.02  & 4.94e-5 & 2.00 \\
$512 \times  512\times  512$   &1.57e-5  &  2.00  & 1.24e-5 & 1.99 \\
\hline
\end{tabular}
\end{center}
\label{tabspheremax}
\end{table}
Tables \ref{tabsphere} and \ref{tabspheremax} show that the convergence 
rates  are of second-order in both discrete  $\ell^2$-norm and discrete 
maximum norm respectively, again confirming the theoretical analysis.  The numerical solution is shown in Fig. \ref{Ex:3}.

\begin{figure}[!ht]
\centering
\subfigure[the velocity field $u^{(1)}$]{
\includegraphics[width=0.35\textwidth]{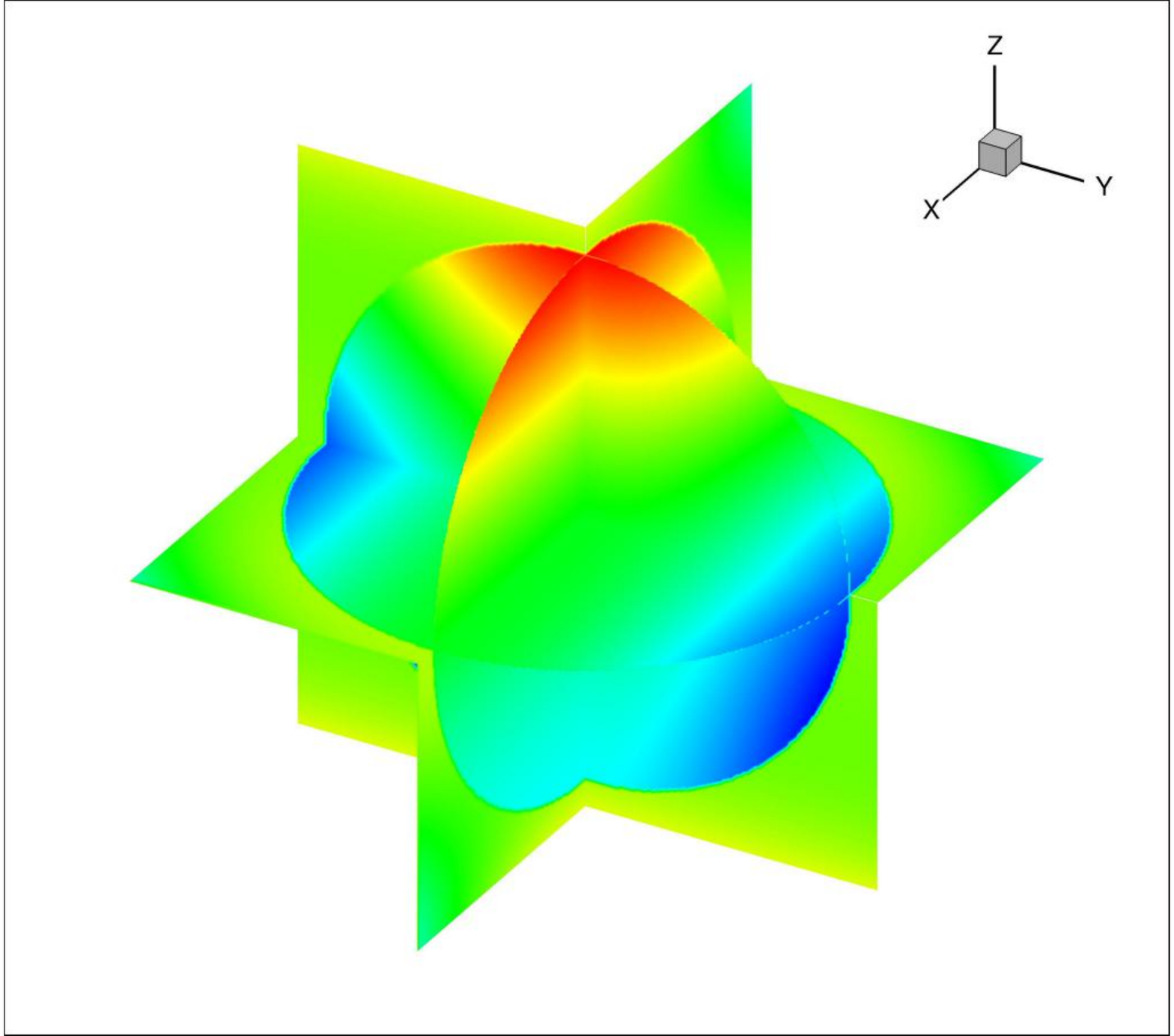}
}
\subfigure[the velocity field $u^{(2)}$]{
\includegraphics[width=0.35\textwidth]{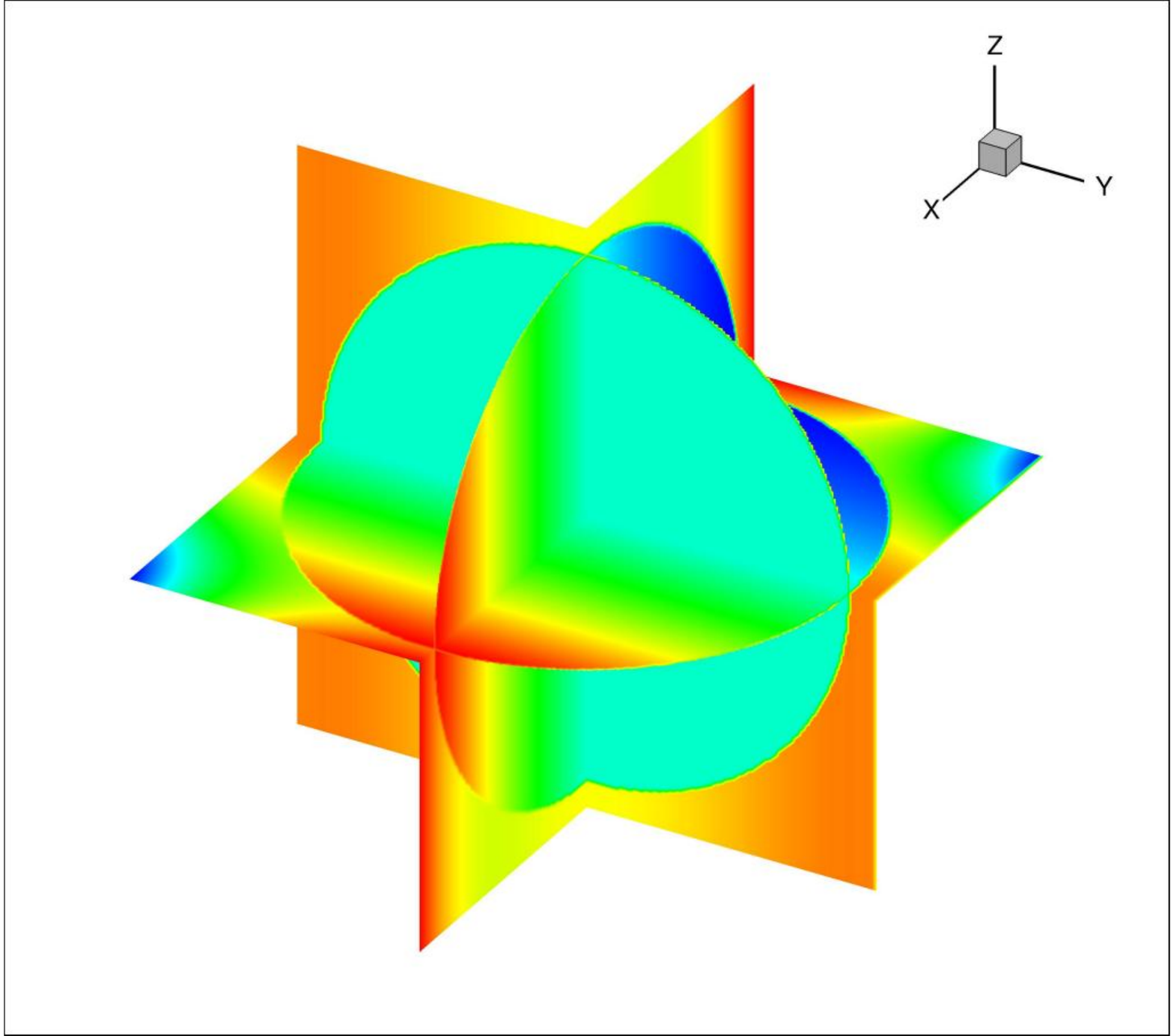}
}
\subfigure[the velocity  field $u^{(3)}$]{
\includegraphics[width=0.35\textwidth]{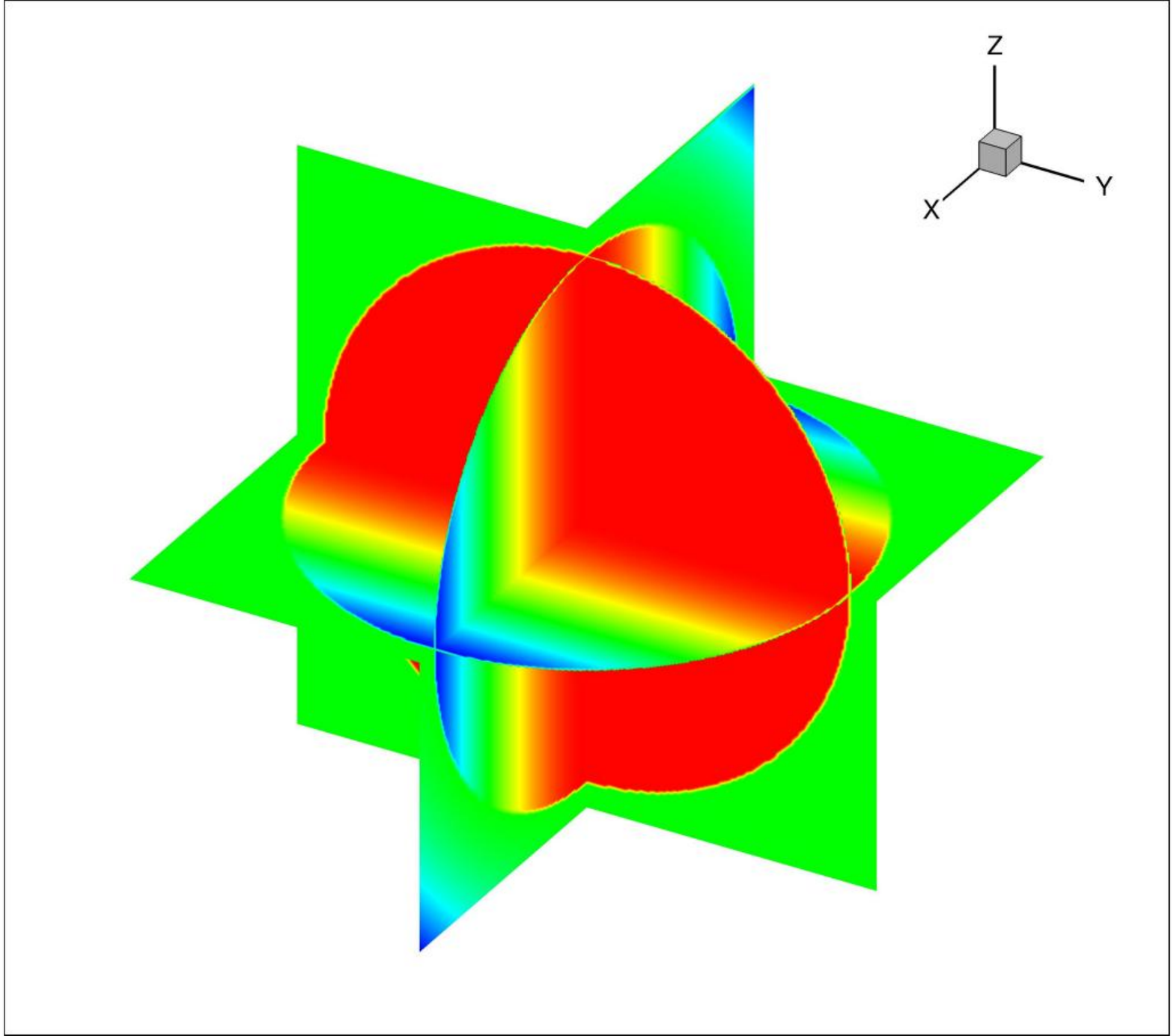}
}
\subfigure[the pressure field $p$]{
\includegraphics[width=0.35\textwidth]{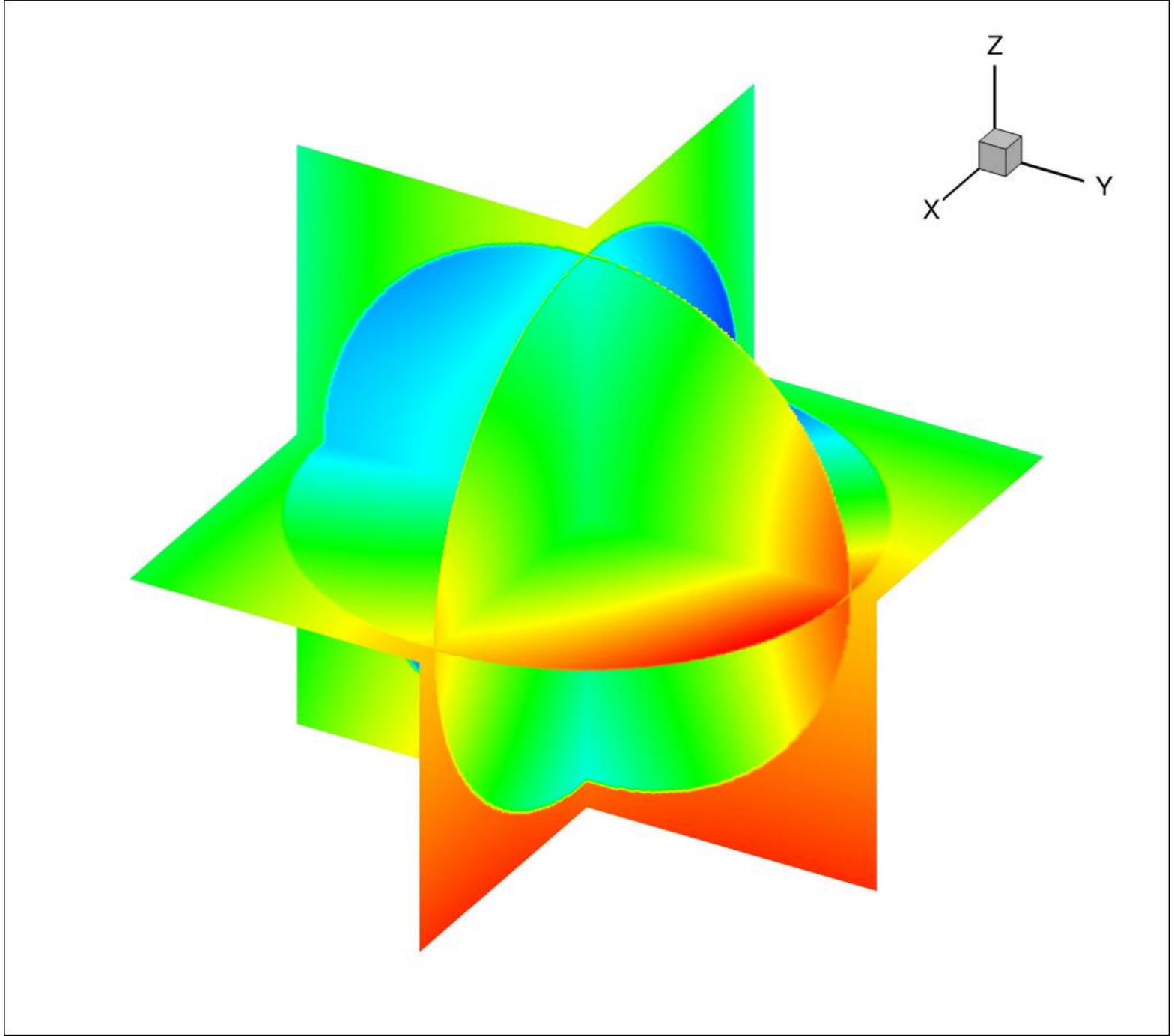}
}
 \setlength{\abovecaptionskip}{-0.0cm}
\setlength{\belowcaptionskip}{-0.0cm}
\caption{Numerical solution for example 3 on a $256\times 256$ grid.} 
 \label{Ex:3}
\end{figure}
\section{Conclusions}
In this work, the second order accuracy of an MAC scheme  for the incompressible Stokes interface problem 
with constant viscosity is proved. Some discrete auxiliary functions, which satisfy  
the discrete Stokes equations, the boundary conditions, the jump conditions 
to a high order of accuracy,  play  a key role in the proof. 
Using the discrete auxiliary functions, the difficulties arising from the 
boundary conditions and the interface are overcome.  
The theoretical results are verified by both 2D and 
3D numerical examples. The numerical experiments also demonstrate that the scheme has second order accuracy 
in the discrete maximum norm for velocity and its gradient, and its theoretical analysis can be obtained similarly as that in \cite{dong2020Maximum}.

\section*{Acknowledgement}
Haixia Dong is partially supported by NSFC under Grant NO. 12001193, the Scientific Research Fund of Hunan Provincial Education Department  (No.20B376), Changsha Municipal Natural Science Foundation (No. kq2014073).
 Wenjun Ying is  partially supported by the Strategic Priority Research Program of Chinese Academy of Sciences (Grant No. XDA25010405), the National Natural Science Foundation of China (Grant No. DMS-11771290) and the Science Challenge Project of China (Grant No. TZ2016002).
Jiwei Zhang is partially supported by NSFC under grant No. 12171376, 2020-JCJQ- ZD-029 and NSAF U1930402.



\begin{thebibliography}{10}
\bibitem{beale2007accuracy}
{\sc T.~Beale, A.~Layton}, {\em On the accuracy of finite difference methods for elliptic problems with interfaces}, Communications in Applied Mathematics and Computational Science, 2007, 1(1): 91-119.

\bibitem{blanc1999error}
{\sc P.~Blanc}, {\em Error estimate for a finite volume scheme on a MAC 
mesh for the Stokes problem}, Finite Volumes for Complex Applications II,  
(1999), pp.~117--124.

\bibitem{blanc2005convergence}
{\sc P.~Blanc}, {\em Convergence of a finite volume scheme on a MAC mesh 
for the Stokes problem with right hand side in $H^{-1}$}, Finite Volumes for 
Complex Applications IV,  (2005), pp.~133--142.

\bibitem{Griffith2009preconditioner}
{\sc E.~G. Boyce}, {\em An accurate and efficient method for the incompressible
  Navier-Stokes equations using the projection method as a preconditioner},
  Journal of Computational Physics, 228 (2009), pp.~7565--7595.

\bibitem{chen2015convergence}
{\sc L.~Chen, M.~Wang, and L.~Zhong}, {\em Convergence analysis of triangular
  MAC schemes for two dimensional Stokes equations}, Journal of Scientific
  Computing, 63 (2015), pp.~716--744.

\bibitem{Chen2018A}
{\sc X.~Chen, Z.~Li, and J.~R. Álvarez}, {\em A direct IIM approach for
  two-phase Stokes equations with discontinuous viscosity on staggered grids},
  Computers \& Fluids,  (2018).

\bibitem{Christoph1990FFT}
{\sc B.~Christoph}, {\em Domain imbedding methods for the Stokes equations},
  Numerical Mathematics, 57 (1990), pp.~435--451.

\bibitem{dong2016unfitted}
{\sc H.~Dong, B.~Wang, Z.~Xie, and L.-L. Wang}, {\em An unfitted hybridizable
  discontinuous Galerkin method for the Poisson interface problem and its error
  analysis}, IMA Journal of Numerical Analysis, 37 (2017), pp.~444--476.
  
\bibitem{dong2020Maximum}
{\sc H.~Dong, W.~Ying, and J. Zhang},   {\em Maximum error estimates of a MAC 
scheme for Stokes equations with Dirichlet boundary conditions}, Applied 
Numerical Mathematics, 150 (2020), pp.~149--163.
  
\bibitem{eymard2010convergence}
{\sc R.~Eymard, T.~Gallou{\"e}t, R.~Herbin, and J.-C. Latch{\'e}}, {\em
  Convergence of the MAC scheme for the compressible Stokes equations}, 
  SIAM Journal on Numerical Analysis, 48 (2010), pp.~2218--2246.

\bibitem{Herbin2012W1}
{\sc T.~Gallou{\"e}t, R.~Herbin, and J.-C. Latch{\'e}}, {\em $w^{1,q}$
  stability of the Fortin operator for the MAC scheme}, Calcolo, 49 (2012), 
  pp.~63--71.

\bibitem{gallouet2016convergence}
{\sc T.~Gallou{\"e}t, R.~Herbin, J.-C. Latch{\'e}, and K.~Mallem}, {\em
  Convergence of the Marker-and-Cell scheme for the incompressible
  Navier--Stokes equations on non-uniform grids}, Foundations of 
  Computational Mathematics,  (2016), pp.~1--41.

\bibitem{gallouet2017convergence}
{\sc T.~Gallou{\"e}t, R.~Herbin, D.~Maltese, and A.~Novotny}, {\em 
Convergence of the Marker-and-Cell scheme for the semi-stationary 
compressible Stokes problem}, Mathematics and Computers in Simulation, 
137 (2017), pp.~325--349.

\bibitem{girault1996finite}
{\sc V.~Girault and H.~Lopez}, {\em Finite-element error estimates for the MAC
  scheme}, IMA Journal of Numerical Analysis, 16 (1996), pp.~347--379.

\bibitem{han1998new}
{\sc H.~Han and X.~Wu}, {\em A new mixed finite element formulation and the 
MAC method for the Stokes equations}, SIAM Journal on Numerical Analysis, 
35 (1998), pp.~560--571.

\bibitem{hansbo2014cut}
{\sc P.~Hansbo, M.~G. Larson, and S.~Zahedi}, {\em A cut finite element method
  for a Stokes interface problem}, Applied Numerical Mathematics, 85 (2014),
  pp.~90--114.

\bibitem{MR3504551}
{\sc X.~He, J.~Li, Y.~Lin, and J.~Ming}, {\em A domain decomposition method for
  the steady-state {N}avier-{S}tokes-{D}arcy model with {B}eavers-{J}oseph
  interface condition}, SIAM Journal on Scientific Computing, 37 (2015), 
  pp.~S264--S290.

\bibitem{MR1163348}
{\sc T.~Y. Hou and B.~T.~R. Wetton}, {\em Convergence of a finite difference
  scheme for the {N}avier-{S}tokes equations using vorticity boundary
  conditions}, SIAM Journal on Numerical Analysis, 29 (1992), pp.~615--639.

\bibitem{MR1220643}
{\sc T.~Y. Hou and B.~T.~R. Wetton}, {\em Second-order convergence of a
  projection scheme for the incompressible {N}avier-{S}tokes equations with
  boundaries}, SIAM Journal on Numerical Analysis, 30 (1993), pp.~609--629.

\bibitem{Hu2018IIM}
{\sc R.~Hu, and Z.~Li}, {\em Error analysis of the immersed interface method for Stokes equations with an interface}, Applied Mathematics Letters, 83 (2018), pp.~207-211.

\bibitem{Peters2005fast}
{\sc V.~R. A.~R. J\''org, Peters}, {\em Fast iterative solvers for discrete
  Stokes equations}, SIAM Journal on Scientific Computing, 27 (2005),
  pp.~646--666.

\bibitem{kanschat2008divergence}
{\sc G.~Kanschat}, {\em Divergence-free discontinuous Galerkin schemes for the
  Stokes equations and the MAC scheme}, International Journal for Numerical
  Methods in Fluids, 56 (2008), pp.~941--950.

\bibitem{leb1964immersed}
{\sc V.~L. Lebedev}, {\em Difference analogues of orthogonal decompositions, fundamental differential operators and certain boundary-value problems of mathematical physics}, Zh. Vychisl. Mat. Mat. Fiz., 4 (1964), pp. 449–465.


\bibitem{lee2003immersed}
{\sc L.~Lee and R.~J. LeVeque}, {\em An immersed interface method for
  incompressible Navier--Stokes equations}, SIAM Journal on Scientific
  Computing, 25 (2003), pp.~832--856.

\bibitem{leveque1994immersed}
{\sc R.~J. Leveque and Z.~Li}, {\em The immersed interface method for elliptic
  equations with discontinuous coefficients and singular sources}, SIAM Journal
  on Numerical Analysis, 31 (1994), pp.~1019--1044.

\bibitem{leveque1997immersed}
{\sc R.~J. LeVeque and Z.~Li}, {\em Immersed interface methods for Stokes flow
  with elastic boundaries or surface tension}, SIAM Journal on Scientific
  Computing, 18 (1997), pp.~709--735.

\bibitem{li2015superconvergence}
{\sc J.~Li and S.~Sun}, {\em The superconvergence phenomenon and proof of the
  MAC scheme for the Stokes equations on non-uniform rectangular meshes},
  Journal of Scientific Computing, 65 (2015), pp.~341--362.

\bibitem{Li2018dependent}
{\sc X.~Li and H.~Rui}, {\em Stability and superconvergence of MAC schemes for time dependent Stokes equations on nonuniform grids,} Journal of Mathematical Analysis and Applications, 466 (2018), pp.~1499-1524.

\bibitem{Li2018NS}
{\sc X.~Li and H.~Rui}, {\em Superconvergence of characteristics marker and cell scheme for the Navier--Stokes equations on nonuniform grids,} SIAM Journal on Numerical Analysis, 56 (2018), pp.~1313-1337.

\bibitem{MR1860918}
{\sc Z.~Li and K.~Ito}, {\em Maximum principle preserving schemes for interface
  problems with discontinuous coefficients}, SIAM Journal on Scientific
  Computing, 23 (2001),
  pp.~339--361.

\bibitem{li2006immersed}
{\sc Z.~Li and K.~Ito}, {\em The immersed interface method: numerical solutions
  of PDEs involving interfaces and irregular domains}, vol.~33, SIAM, 2006.

\bibitem{MR3623206}
{\sc Z.~Li, H.~Ji, and X.~Chen}, {\em Accurate solution and gradient
  computation for elliptic interface problems with variable coefficients}, SIAM
  Journal on Numerical Analysis, 55 (2017), pp.~570--597.

\bibitem{li2007immersed}
{\sc Z.~Li, M.-C. Lai, and K.~Ito}, {\em An immersed interface method for the
  Navier-Stokes equations on irregular domains}, PAMM: Proceedings in Applied Mathematics and Mechanics, 7 (2007),
  pp.~1025401--1025402.

\bibitem{Li2015Analysis}
{\sc Z.~Li, L.~Wang, E.~Aspinwall, R.~Cooper, P.~Kuberry, A.~Sanders, and K.~Zeng} {\em Some new analysis results for a class of interface problems}, Mathematical Methods in the Applied Sciences, 38 (2015), pp,~4530-4539.


\bibitem{MR3788553}
{\sc Z.~Li, M.-C. Lai, X.~Peng, and Z.~Zhang}, {\em A least squares augmented
  immersed interface method for solving {N}avier-{S}tokes and {D}arcy coupling
  equations}, Comput. \& Fluids, 167 (2018), pp.~384--399.

\bibitem{MR736332}
{\sc A.~Mayo}, {\em The fast solution of {P}oisson's and the biharmonic
  equations on irregular regions}, SIAM Journal
  on Numerical Analysis, 21 (1984),
  pp.~285--299.

\bibitem{MR1145178}
{\sc A.~Mayo and A.~Greenbaum}, {\em Fast parallel iterative solution of
  {P}oisson's and the biharmonic equations on irregular regions},  SIAM Journal 
  on Scientific and Statistical Computing, 13 (1992), pp.~101--118.

\bibitem{mori2008convergence}
{\sc Y.~Mori}, {\em Convergence proof of the velocity field for a Stokes flow
  immersed boundary method}, Communications on Pure and Applied Mathematics, 
  61 (2008), pp.~1213--1263.

\bibitem{nicolaides1996analysis}
{\sc R.~Nicolaides and X.~Wu}, {\em Analysis and convergence of the MAC scheme.
  II. Navier-Stokes equations}, Mathematics of Computation of the American
  Mathematical Society, 65 (1996), pp.~29--44.

\bibitem{nicolaides1992analysis}
{\sc R.~A. Nicolaides}, {\em Analysis and convergence of the MAC scheme. I. the
  linear problem}, SIAM Journal on Numerical Analysis, 29 (1992),
  pp.~1579--1591.

\bibitem{peskin2002immersed}
{\sc C.~S. Peskin}, {\em The immersed boundary method}, Acta Numerica, 11
  (2002), pp.~479--517.

\bibitem{rui2017stability}
{\sc H.~Rui and X.~Li}, {\em Stability and superconvergence of MAC scheme for
  Stokes equations on nonuniform grids}, SIAM Journal on Numerical Analysis, 55
  (2017), pp.~1135--1158.


\bibitem{rui2020Darcy}
{\sc H.~Rui and Y.~Sun}, {\em A MAC  scheme for coupled Stokes–Darcy equations on non-uniform grids}, Journal of Scientific Computing, 82 (2020), pp.~1-29.

\bibitem{rutka2008staggered}
{\sc V.~Rutka}, {\em A staggered grid-based explicit jump immersed interface
  method for two-dimensional Stokes flows}, International Journal for Numerical
  Methods in Fluids, 57 (2008), pp.~1527--1543.

\bibitem{Saad1993GMRES}
{\sc Y.~Saad}, {\em A flexible inner-outer preconditioned GMRES algorithm},
  SIAM Journal on Scientific Computing, 14 (1993), pp.~461--469.

\bibitem{shibata2003resolvent}
{\sc Y.~Shibata and S.~Shimizu}, {\em On a resolvent estimate of the interface
  problem for the Stokes system in a bounded domain}, Journal of Differential
  Equations, 191 (2003), pp.~408--444.

\bibitem{shin1997inf}
{\sc D.~Shin and J.~C. Strikwerda}, {\em Inf-sup conditions for
  finite-difference approximations of the Stokes equations}, The ANZIAM
  Journal, 39 (1997), pp.~121--134.

\bibitem{stenberg1989some}
{\sc R.~Stenberg}, {\em Some new families of finite elements for the Stokes
  equations}, Numerische Mathematik, 56 (1989), pp.~827--838.

\bibitem{MR0166942}
{\sc G.~Strang}, {\em Accurate partial difference methods.}, Numerische Mathematik, 
6 (1964), pp.~37--46.

\bibitem{tan2008immersed}
{\sc Z.~Tan, D.~V.~Le, Z.~Li, K.~Lim, and B.~Khoo}, {\em An immersed interface method for solving incompressible viscous flows with piecewise constant viscosity across a moving elastic membrane}, Journal of Computational Physics, 
227 (2008), pp.~9955--9983.
  
\bibitem{tan2009immersed}
{\sc Z.~Tan, D.~V.~Le, K.~Lim, and B.~Khoo}, {\em An immersed interface method for the incompressible Navier–Stokes equations with discontinuous viscosity across the interface}, SIAM Journal on Scientific Computing, 
31 (2009), pp.~1798--1819.  
  
\bibitem{tan2011implementation}
{\sc Z.~Tan, K.~Lim, and B.~Khoo}, {\em An implementation of MAC grid-based
  IIM-Stokes solver for incompressible two-phase flows}, Communications in
  Computational Physics, 10 (2011), pp.~1333--1362. 
  

\bibitem{thomas1977mixed}
{\sc J.~Thomas and P.~Raviart}, {\em A mixed finite element method for 2nd
  order elliptic problems}, Mathematical aspects of finite element methods, (1977),
  pp.~292--315.

\bibitem{Tong2020gradient}
{\sc F.~Tong, W.~Wang, X.~Feng, J.~Zhao, and Z.~Li} {\em How to obtain an accurate gradient for interface problems?}, Journal of Computational Physics, 405 (2020), pp.109070. 
 
\bibitem{wang2013hybridizable}
{\sc B.~Wang and B.~Khoo}, {\em Hybridizable discontinuous Galerkin method
  (HDG) for Stokes interface flow}, Journal of Computational Physics, 247
  (2013), pp.~262--278.

\bibitem{wang2015new}
{\sc Q.~Wang and J.~Chen}, {\em A new unfitted stabilized Nitsche’s finite
  element method for Stokes interface problems}, Computers \& Mathematics with
  Applications, 70 (2015), pp.~820--834.
  
\bibitem{wang2022simple}
{\sc W.~Wang and Z.~Tan}, {\em A simple augmented IIM for 3D incompressible two-phase Stokes flows with interfaces and singular forces},
 Computer Physics Communications, 270(2022), pp.108154.

\end{thebibliography}


\end{document}